\documentclass[11pt, a4paper]{article}
\usepackage[all]{xy}
\usepackage[latin1]{inputenc} 
\usepackage[T1]{fontenc}
\usepackage[dvips]{graphics,graphicx}
\usepackage{amsfonts}
\usepackage{amssymb}
\usepackage{amsrefs}
\usepackage{amsmath}
\usepackage{mathtools}
\usepackage{esint}
\usepackage{algorithm, algorithmicx}
\usepackage{algpseudocode}
\usepackage{makecell}

\allowdisplaybreaks

\usepackage{amsbsy,
	amsopn,
	amscd, 
	amsxtra, 
	amsthm,
	verbatim}
\usepackage{upref}
\usepackage{xcolor}
\usepackage{enumitem}
\usepackage[colorlinks,
linkcolor=blue,
anchorcolor=blue,
citecolor=blue
]{hyperref}

\usepackage{thmtools}
\usepackage{caption}
\usepackage{subcaption}
\usepackage{booktabs}
\usepackage{longtable}
\usepackage[margin=1in]{geometry}

\newtheorem{lemma}{Lemma}[section]
\newtheorem{theorem}[lemma]{Theorem}
\newtheorem{proposition}[lemma]{Proposition}
\newtheorem*{proposition*}{Proposition}

\newtheorem*{corollary*}{Corollary}
\newtheorem{definition}[lemma]{Definition}
\newtheorem{assumption}{Assumption}
\theoremstyle{remark}
\newtheorem{remark}[lemma]{Remark}
\newtheorem{example}[lemma]{Example}

\newcommand{\ud}{\,\mathrm{d}}
\newcommand{\R}{\mathbb{R}}
\newcommand{\T}{\mathbb{T}}
\newcommand{\pb}{\mathbb{P}}
\newcommand{\calK}{\mathcal{K}}
\newcommand{\calF}{\mathcal{F}}
\newcommand{\fve}{\mathcal{F}_\varepsilon}
\newcommand{\wvet}{w^{(\varepsilon)}_t}
\newcommand{\Kve}{K_\varepsilon}
\newcommand{\xieps}{\xi_\varepsilon}
\newcommand{\opQ}{\mathcal{Q}}
\newcommand{\rhove}{\rho^{(\varepsilon)}}

\newcommand{\Lambdae}{\Lambda_\varepsilon}
\DeclareMathOperator*{\argmax}{arg\,max}
\DeclareMathOperator*{\argmin}{arg\,min}
\DeclareMathOperator*{\KL}{KL}
\newlength{\leftstackrelawd}
\newlength{\leftstackrelbwd}
\def\leftstackrel#1#2{\settowidth{\leftstackrelawd}%
{${{}^{#1}}$}\settowidth{\leftstackrelbwd}{$#2$}%
\addtolength{\leftstackrelawd}{-\leftstackrelbwd}%
\leavevmode\ifthenelse{\lengthtest{\leftstackrelawd>0pt}}%
{\kern-.5\leftstackrelawd}{}\mathrel{\mathop{#2}\limits^{#1}}}

\usepackage{microtype}
\definecolor{darkgreen}{rgb}{0.1,0.6,0.1}
\definecolor{darkred}{rgb}{0.6,0,0}
\definecolor{lightgray}{rgb}{0.5,0.5,0.5}
\newcommand{\Yu}[1]{{\color{darkgreen}#1}}
\newcommand{\De}[1]{{\color{orange}#1}}
\newcommand{\Li}[1]{{\color{blue}#1}}
\newcommand{\ed}[1]{{\color{red}#1}}
\newcommand{\wkstc}{\overset{\ast}{\rightharpoonup}}
 \newenvironment{listi}
  {\begin{list} 
 {(\roman{broj})}
{ \usecounter{broj}}
     \setlength{\labelwidth}{25pt}
  }
{   \end{list} }
 \newenvironment{lista}
  {\begin{list} 
 {(\alph{broja})}
{ \usecounter{broja}}
     \setlength{\labelwidth}{25pt}
  }
{   \end{list} }
\newcounter{broj}
\newcounter{broja}

\newcommand{\nc}{\normalcolor}

\title{Birth-death dynamics for sampling: Global convergence, approximations and their asymptotics}
\author{Yulong Lu\thanks{School of Mathematics, University of Minnesota, Twin Cities, Minneapolis, MN, 55455, USA.  (yulonglu@umn.edu)}
	\and Dejan Slep\v{c}ev\thanks{Department of Mathematical Sciences, Carnegie Mellon University, Pittsburgh, PA, 15213, USA (slepcev@math.cmu.edu)}
	\and Lihan Wang\thanks{(Corresponding Author) Department of Mathematical Sciences, Carnegie Mellon University, Pittsburgh, PA, 15213, USA (lihanw@andrew.cmu.edu)}}
\date{\today}
\providecommand{\keywords}[1]{\textbf{\textit{Keywords:}} #1}

\begin{document}

\maketitle


\begin{abstract}
    Motivated by the challenge of sampling Gibbs measures with nonconvex potentials, we study a continuum birth-death dynamics. We improve results in previous works \cite{lu2019accelerating, liu2022polyak} and provide weaker hypotheses under which the probability density of the birth-death governed by Kullback-Leibler divergence or by $\chi^2$ divergence converge exponentially fast to the Gibbs equilibrium measure, with a universal rate that is independent of the potential barrier. To build a practical numerical sampler based on the pure birth-death dynamics, we consider an interacting particle system, which is inspired by the gradient flow structure and the classical Fokker-Planck equation and relies on kernel-based approximations of the measure. Using the technique of $\Gamma$-convergence of gradient flows,   
    we show  that on the torus, smooth and bounded positive solutions of the kernelized dynamics converge on finite time intervals, to the pure birth-death dynamics as the kernel bandwidth shrinks to zero. 
    Moreover we provide quantitative estimates on the bias of minimizers of the energy corresponding to the kernelized dynamics.
    Finally we prove the long-time asymptotic results on the convergence of the 
    asymptotic states of the kernelized dynamics towards the Gibbs measure. 
\end{abstract}

\keywords{spherical Hellinger metric; gradient flow; statistical sampling; birth-death dynamics.}

\section{Introduction}
Sampling from a given target probability distribution has diverse applications, including 
 Bayesian statistics, machine learning, statistical physics, and many others. In practice, the measure $\pi$ is often the Gibbs measure corresponding to potential $V: \R^d \rightarrow \R$: 
\[
\pi(x) = \frac{1}{Z} e^{-V(x)}, \qquad \textrm{for } x \in \R^d,
\]
where $Z$ is the typically unknown normalization constant.
\smallskip
 
 Some of the most popular methods for sampling such distributions are based on 
 Markov chain Monte Carlo (MCMC) approach. Much of the research works on MCMC have been devoted to designing Markov chains that are ergodic with respect to the target probability measure and enjoy fast mixing properties. Popular sampling methods include Langevin MCMC \cite{grest1986molecular,roberts1996exponential}, Hamiltonian Monte Carlo \cite{andersen1980molecular,neal2011mcmc,rossky1978brownian, bou2017randomized}, bouncy particle and zigzag samplers \cite{bierkens2019zig,bouchard2018bouncy}, and affine-invariant ensemble MCMC \cite{goodman2010ensemble}, and Stein variational gradient descent (SVGD) \cite{liu2016stein}.
When the potential function $V$ is strongly convex, these sampling methods perform quite well; we refer to recent literature \cite{dalalyan2017theoretical,durmus2017nonasymptotic,vempala2019rapid,wibisono2018sampling,lu2022explicit} and references therein for understanding their convergence and computational complexity.
However, the efficiency of these sampling methods are hampered by the multi-modality of $\pi$ (corresponding to a non-convex $V$) as it takes exponentially long time for the sampler to hop from one mode to another. Many diffusion-based samplers suffer from such metastability issue, and numerous techniques have been proposed to alleviate this issue, including in particular parallel and simulated tempering \cite{swendsen1986replica,neal2001annealed,marinari1992simulated} and adaptive biasing methods \cite{wang2001efficient,laio2002escaping, darve2001calculating, henin2004overcoming, lelievre2008long}. 
 \smallskip
 
The sampling problem can be recast as an optimization problem  on the space of probability measures \cite{wibisono2018sampling,bernton2018langevin}. Indeed, inspired by the seminal work of Jordan, Kinderlehrer and Otto \cite{jordan1998variational}, the Fokker-Planck equation associated to the (overdamped) Langevin dynamics can be viewed as the Wasserstein gradient flow of the KL-divergence \begin{equation*}
    \calF(\rho) := \KL(\rho|\pi) = \int_{\R^d}  \log \frac{\rho}{\pi}\ud \rho,
\end{equation*} 
which suggests that the Langevin dynamics can be seen as the steepest descent flow of the KL-divergence, along which the initial distribution flows towards the target distribution.
The gradient-flow perspective provides a way towards building new sampling dynamics by designing new objective functions or new metrics for the manifolds of probability measures. For example, the underdamped Langevin dynamics can be viewed as a Nesterov's accelerated  gradient descent on the space of probability measures \cite{ma2019there}.

\smallskip
 
Langevin dynamics converge exponentially fast to the Gibbs measure under the assumption that the target measure satisfies the Logarithmic Sobolev inequality \cite{bakry2014analysis}. Unfortunately, the convergence rate is limited by the optimal constant in the Log-Sobolev inequality, which can be very small when the target measure is multimodal. This  is because it can take exponentially long time for the dynamics to overcome the energy barriers and hop between the multiple modes.

\subsubsection{Birth-death dynamics, and their long-time convergence} 
The issues described above prompt the following question:

 {\em Can one construct a gradient-flow  dynamics for sampling that achieves a potential-independent convergence rate?}
  \smallskip
  
Recent work \cite{lu2019accelerating} by Lu, Nolen and one of the authors  gives an affirmative answer to the question by proposing the following birth-death dynamics for sampling 
\begin{equation}\label{eqn:purebd}\tag{BD}
    \partial_t \rho_t = -\rho_t \log \frac{\rho_t}{\pi}+ \rho_t \int_{\R^d} \rho_t \log \frac{\rho_t}{\pi}\ud x.
\end{equation} 
An equation of similar type on discrete space, known as Replicator equation, also appears in information geometry literature, see \cite[Chapter 6]{ay2017information} and references therein. Note that the dynamics \eqref{eqn:purebd} is agnostic to the normalization constant. It has been shown that \eqref{eqn:purebd} is a gradient flow of the KL-divergence with respect to 
 the spherical Hellinger distance\footnote{In information geometry literature \cite{amari2016information,ay2017information}, the spherical Hellinger distance is also known as the Fisher-Rao distance. On the other hand, in other works, for example \cite{chizat2018interpolating}, the terminology ``Fisher-Rao distance'' refers to the Hellinger distance, which is defined on positive measures. To avoid confusion and emphasize the fact that these two distances are defined on different spaces, we avoid using ``Fisher-Rao distance'' altogether in this work.} $d_{SH}$ defined by  \eqref{eqn:conicdSH} (see \cite{laschos2019geometric}*{Section 3}). 
Furthermore the authors show an exponential rate of convergence for initial data such that $\rho_0$ is bounded below by a positive multiple of $\pi$ and 
$\KL(\rho_0 | \pi) \le 1$. More recent work \cite{liu2022polyak}  established exponential convergence of \eqref{eqn:purebd} if $\frac{\rho_0}{\pi}$ is bounded from above and below.

\smallskip


In Theorem \ref{thm:expconvbdkl} we improve both results by showing that the solution $\rho_t$ contracts to the equilibrium with a uniform (potential-independent) rate from any $\rho_0$ that is only bounded from below, but not necessarily above with respect to $\pi$. The condition $\KL(\rho_0|\pi) \le 1$ is no longer required. This has an important  practical consequence, namely it shows that it is sufficient to start the dynamics with an initial density that is more spread out than $\pi$, which is easy to guarantee for many target measures $\pi$. The removal of upper bound also allows us to choose a sufficient wide round Gaussian to satisfy the pointwise lower bound for a large class of $\pi$ in $\R^d$.

\smallskip

In Section \ref{sec:chi} we also investigate birth-death dynamics
\begin{equation}\label{eqn:gfchi2} \tag{BD2}
    \partial_t \rho_t =-\rho_t\left( \frac{\rho_t}{\pi}- \int \frac{\rho_t}{\pi}\ud \rho_t\right)
\end{equation}
that arises as the spherical Hellinger gradient flow of $\chi^2$-divergence which is at the basis of the algorithm proposed in \cite{lindsey2021ensemble}. In particular we prove that the $\chi^2$-divergence between the dynamics and the target measure converges to zero exponentially fast, with a rate independent of the potential, see Theorem \ref{thm:expconvbdchi2}. This complements the result of \cite{lindsey2021ensemble}, which proved the convergence of the reverse $\chi^2$-divergence.


\subsubsection{Approximations to \eqref{eqn:purebd} that allow for discrete measures}
Since the equation \eqref{eqn:purebd} is not well-defined when $\rho$ is a discrete measure, it is unclear whether one can build interacting particle sampling schemes directly based on it. 
A principled way to build particle approximations to  \eqref{eqn:purebd} and \eqref{eqn:gfchi2} is to define new dynamics which approximate \eqref{eqn:purebd} and \eqref{eqn:gfchi2} and are well-defined for discrete measures. A further desirable property for these dynamics is to retain the (spherical) Hellinger gradient flow  structure.

\smallskip

In doing so we are inspired by the work of Carrillo, Craig, and Patacchini \cite{carrillo2019blob}, who considered the  Wasserstein gradient flow of the regularized KL-divergence as an approximation of the Fokker-Planck equation. In particular they introduced the regularized energy
\begin{equation*}
    \fve(\rho) = \int \rho \log (\Kve * \rho) - \int \rho\log \pi= \int \rho \log(\Kve * \rho) + \int \rho V. 
\end{equation*}

The gradient flow of $\fve$ with respect to Wasserstein metric, studied in \cite{carrillo2019blob}, is
\begin{equation} \label{eq:blob}
\partial_t \rhove_t = \nabla \cdot \left(\rho \nabla  \frac{\delta \fve}{\delta \rho}\right),
\end{equation}
 where $ \frac{\delta \fve}{\delta \rho}$ is the functional derivative of  $\fve$ defined by 
\begin{equation}\label{eqn:funcderivkl}
    \frac{\delta \fve}{\delta \rho} = \log \left(\frac{\Kve*\rho}{\pi} \right) + \Kve * \left( \frac{ \rho}{\Kve* \rho} \right). 
\end{equation}
For  discrete  initial data, i.e. $\rho_0 = \sum m_i \delta_{x_i}$, the solution remains a discrete measure for any positive time, where the evolution of particles is given by a system of ordinary differential equations.  It heuristically provides a deterministic particle approximation of the Fokker-Planck equation, though rigorous convergence analysis remains open.
\smallskip

The gradient flow of $\fve$ with respect to the spherical Hellinger distance is the equation
\begin{equation} \label{eq:gfeps} \tag{$\text{BD}_\varepsilon$}
 \!   \partial_t \rhove_t = -\rhove_t \left[\log \left(\frac{\Kve*\rhove_t}{\pi}\right) + \Kve * \left( \frac{ \rhove_t}{\Kve* \rhove_t} \right) - \! \int \log \left(\frac{\Kve*\rhove_t}{\pi}\right) \rhove_t -1 \right].
\end{equation}
Note that the right hand side is well-defined even if $\rhove$ is a discrete measure. This suggests the possibility of approximating   \eqref{eqn:gfkerbd} with interacting particles. Indeed, we will introduce and discuss in Section \ref{sec:jump} a particle-based jump process whose mean-field limit is heuristically characterized by equation \eqref{eqn:gfkerbd}  and present  some numerical experiments on its application for sampling in Section \ref{sec:numerics}. 
\smallskip

\emph{Bias of global minimizers of $\fve$.}
For sufficiently small $\varepsilon$, we expect $\fve$ to be only a small perturbation of $\calF$, and hence the global minimizers of $\fve$ should also be a perturbation of $\pi$.
In Section \ref{sec:minimizer-approx} we prove that such bias is at most of order $\varepsilon$, improving on the qualitative $\Gamma$-convergence result in \cite{carrillo2019blob}. More precisely, we show that for any minimizer $\pi_\varepsilon$ of $\fve$, the Wasserstein distance between $\pi_\varepsilon$ and $\pi$ is no more than the order  $O(\varepsilon)$. The optimality of such upper bound is demonstrated by numerical experiments. 

\smallskip

\emph{Convergence of the dynamics, on finite time intervals.}
In Section \ref{sec:gamma} we investigate the convergence of the solutions of 
\eqref{eq:gfeps} towards solutions of the pure birth-death dynamics \eqref{eqn:purebd}.
More precisely we use the $\Gamma$-convergence of gradient flows to show that on arbitrary finite time intervals,
smooth solutions of \eqref{eq:gfeps} with initial condition $\rho_0$ bounded below on a bounded domain converge
 to solutions of \eqref{eqn:purebd} as $\varepsilon \rightarrow 0$.  For unbounded domains there are substantial difficulties to handle the decay of $\pi$ at infinity, and proving $\Gamma$-convergence of gradient flows in such setting remains an open problem. 
\smallskip

\emph{Convergence of the dynamics, asymptotic states.}
We note that since $\Gamma$-convergence of gradient flows is in general stated for finite time intervals, and thus does not directly imply that the asymptotic states of the dynamics \eqref{eq:gfeps} converge towards the asymptotic state of \eqref{eqn:purebd}, namely  $\pi$. This is a general issue for the convergence of gradient flows and is of practical interest. Namely in many applications one uses approximate gradient flows with aim to approximate the limiting state of the original gradient flow. In Section \ref{sec:asymptotics} we investigate the relationship between $\Gamma$-convergence of gradient flows and the convergence of asymptotic states. Specifically in Proposition \ref{prop:asyset} we prove convergence of asymptotic states in the general setting of gradient flows in metric spaces. We apply the result in two settings, one is the setting of the gradient flows of this paper (Theorem \ref{thm:bdasysetwkconv}) and the other (Theorem \ref{thm:jmmasyconv}) is the convergence of two-layer neural networks studied in \cite{javanmard2020analysis} , which we discuss at the end of this section. 
In particular, we prove in Theorem  \ref{thm:bdasysetwkconv} that $\pi$ must be the only possible limit of the asymptotic states $\rhove_\infty$ in $W_2$.
 The proof is general and relies on the fact that $\pi$ is the unique minimizer of the KL divergence. 

\smallskip

\emph{Application to convergence of asymptotic states for 2-layer neural networks.}
In \cite{javanmard2020analysis}, the authors considered the problem of learning a strongly concave function $f$ using bump-like neurons where the width of the kernels $\delta\ll 1$. As the number of neurons approach infinity, the process of stochastic gradient descent with noise $\tau$ converges to the Wasserstein gradient flow of the following entropy-regularized risk functional
\begin{equation}\label{eqn:jmmfunctional} F^\delta(\rho^\delta) = \int_{\Omega} \left(\frac{1}{2}(K_\delta*\rho^\delta-f)^2+\tau \rho^\delta \log \rho^\delta\right) \ud x.\end{equation} 
Here $\Omega$ is a smooth, convex and compact domain. 
More precisely, the gradient flow equation writes 
\begin{equation}\label{eqn:jmmgfdelta}
    \partial_t \rho^\delta_t = \nabla \cdot (\rho_t^\delta \nabla \Psi) +\tau \Delta \rho_t^\delta, \, \textrm{ with }\Psi = -K^\delta*f + K^\delta*K^\delta*\rho_t^\delta.
\end{equation} The authors proved that as $\delta\to 0$, with suitable initial and boundary conditions, the solution of \eqref{eqn:jmmgfdelta} converges strongly in $L^2$ to the solution of the limiting gradient flow
\begin{equation}\label{eqn:jmmgf0}
     \partial_t \rho_t = \nabla \cdot (\rho_t \nabla (\rho_t-f)) +\tau \Delta \rho_t,
\end{equation} which is the Wasserstein gradient flow of the limiting functional 
\begin{equation}\label{eqn:jmmfcnl0}
    F(\rho) =\int_{\Omega} \left(\frac{1}{2}(\rho-f)^2+\tau \rho \log \rho\right) \ud x.
\end{equation}
Moreover, since \eqref{eqn:jmmfcnl0} is displacement convex with respect to Wasserstein geodesics, \eqref{eqn:jmmfcnl0} has a unique minimizer and \eqref{eqn:jmmgf0} converges exponentially to that minimizer as $t\to\infty$. The work \cite{javanmard2020analysis}, however, does not provide  results regarding the long-time behavior of the regularized gradient flow \eqref{eqn:jmmgfdelta}, which is the numerically approximated dynamics. Our Proposition \ref{prop:asyset} provides the tools to prove  convergence of limiting states, resulting in Theorem \ref{thm:jmmasyconv} below. For more results regarding the long-time convergence of such dynamics arising from the training of neural networks, we refer the readers to the recent works \cite{chizat2022mean, craig2022blob}.

\smallskip

In Section \ref{sec:numerics} we provide two numerical experiments, one on a toy example of 2-dimensional Gaussian mixture, another on a real-world Bayesian classification problem, to demonstrate the effectiveness of the birth-death algorithm. In both examples we observe that birth-death sampler allows significantly faster mixing of particles compared to Langevin dynamics or SVGD. More specifically, in the multimodal Example \ref{exp:toyGMM}, one can see in Figure \ref{fig:gmmparticles} that, once a high-probability mode is discovered, birth-death sampler will facilitate movement of particles towards this newly discovered mode, which helps overcoming the issue of metastability. In the real-world Example \ref{exp:BayesClass} where the non-convexity is not strong and SVGD works well, birth-death sampler can reach the equilibrium in an extremely short time. 

\subsection{Contributions}
We highlight the major contributions of the present paper as follows: 

\begin{itemize}
    \item We prove that the pure birth-death dynamics \eqref{eqn:purebd} converges globally to its unique equilibrium measure $\pi$  with a uniform rate with respect to KL-divergence, improving the results in \cite{lu2019accelerating, liu2022polyak}. See Theorem \ref{thm:expconvbdkl} for the precise statements. Using similar techniques, we also investigate the algorithm proposed in \cite{lindsey2021ensemble}, and prove that their time-rescaled infinite-particle equation converges in $\chi^2$-divergence  converges exponentially with a rate independent of the potential, see Theorem \ref{thm:expconvbdchi2}.
    \item We show that under suitable conditions, any global minimizer of the regularized energy \eqref{eqn:funckerkl} $\pi_\varepsilon$ is $O(\varepsilon)$ close to $\pi$ under $W_2$ distance. The precise statement can be found in Theorem \ref{thm:quanminconv}.
\item We show in Theorem \ref{thm:bdgammaconv} that smooth solutions of the kernelized dynamics \eqref{eq:gfeps} with densities bounded above and below on torus $\Gamma$-converges to the  pure birth-death dynamics \eqref{eqn:purebd} within any finite time-horizon in the limit of small kernel width. As a corollary, this justifies the convergence of the density $\rhove_{t}$ of the kernelized dynamics with width $\varepsilon$ to the target measure $\pi$ as $\varepsilon \rightarrow 0$ and $t\rightarrow \infty$. 
\item Finally, we show in Theorem \ref{thm:bdasysetwkconv} that on torus, the long-time limit of \eqref{eqn:gfkerbd} converges with respect to Hausdorff distance corresponding to the Wasserstein distance.
\end{itemize} 

\subsection{Related works} 
The pure birth-death dynamics \eqref{eqn:purebd}  is a gradient flow on relative entropy with respect to the the spherical Hellinger distance we  define in \eqref{eqn:conicdSH}. 
The mass-conservative metric \eqref{eqn:conicdSH}  was introduced in \cite{brenier2020optimal, kondratyev2019spherical, laschos2019geometric} and used in the study equations modeling fluids and population dynamics. Later it was applied in the analysis of training process of neural networks \cite{rotskoff2019global,wei2019regularization}. In the context of statistical sampling, the  spherical Hellinger metric was first applied by \cite{lu2019accelerating} to accelerate Langevin dynamics for sampling, where a local exponential convergence (Theorem \ref{thm:oldconv}) was proved. The paper \cite{gabrie2022adaptive} uses the idea of birth-death dynamics to improve the training of normalizing flows that learn the target distribution. The recent paper \cite{lindsey2021ensemble} constructs an ensemble MCMC algorithm  whose mean field evolution is given by the spherical Hellinger gradient flow of the $\chi^2$-divergence (see  \eqref{eqn:gfchi2}). Convergence to the equilibrium was also established  therein on a finite state space. 
The construction of birth-death dynamics \eqref{eqn:purebd} is also related to the recent study on unbalanced optimal transport and associated gradient flows. In particular, the unbalanced transportation metric, called the Hellinger-Kantorovich metric, which interpolates between 2-Wasserstein and Hellinger metric, 
was defined and studied in \cite{kondratyev2016new,liero2018optimal,chizat2018interpolating} and allows for transport between measures with different masses.
\smallskip

Ensemble-based sampling methods have also been widely studied in recent years, which is another important motivation of our work. Ensemble-based sampling allows for global view of the particle configurations and enables for the particles to exchange information.
One of the most successful sampling methods in this category is the affine invariant sampler introduced by Goodman and Weare \cite{goodman2010ensemble}; see also  \cite{foreman2013emcee}. Ensemble-based sampling are also related to sequential Monte Carlo \cite{del2006sequential} and importance sampling \cite{bunch2016approximations, reich2013guided}. In a continuous-time point of view, ensemble-based samplers can be developed via interacting particle systems, examples of which include ensemble Kalman methods \cite{reich2011dynamical, garbuno2020interacting}, consensus based sampling \cite{carrillo2022consensus} (which also has ideas from optimization \cite{pinnau2017consensus, carrillo2018analytical}), and ensemble Langevin dynamics \cite{liu2022second}. 

\smallskip

We are particularly interested in sampling approaches that are defined as gradient flows of a functional measuring the difference from the target measure. There are a variety of functionals and metrics considered. 
Blob particle method \cite{carrillo2019blob} is the Wasserstein gradient flow of the regularization of KL divergence that allows for discrete measures where it becomes an interacting particle system. Such viewpoint has been applied to sampling purposes in \cite{craig2022blob} where the authors considered the Wasserstein gradient flow for regularized $\chi^2$ energy. A different approach to create gradient-flow based interacting-particle systems for sampling is the 
 SVGD introduced in \cite{liu2016stein}. There the authors consider the gradient flow of the standard KL-divergence with respect to a metric (now known as the Stein geometry) which requires smoothness of the velocities. Thus the gradient-flow velocity makes sense even when considered for particle measures  \cite{liu2017stein,lu2019scaling,duncan2019geometry}.
The work \cite{chewi2020svgd} provides a new perspective on SVGD by viewing it as a kernelized gradient flow of the $\chi^2$-divergence.
A further direction of research considers  Wasserstein gradient flows of the distance to the target measure in a very weak metric that is well defined for particles; in particular Kernelized Stein Discrepancy \cite{korba2021kernel}.
  Recently the work \cite{lambert2022variational} considered using Wasserstein gradient flow for variational inference, where they use Gaussian mixtures to approximate the target density with the evolution of mean and variance governed by gradient flows.


\section{Pure birth-death dynamics governed by relative entropy}
Let us first introduce the Benamou-Brenier formulation of the Hellinger distance (the distance plays an important role in information geometry, see for example \cite{amari2016information,ay2017information})
on (not necessarily probability) measures
\begin{equation}\label{eqn:ubfisherrao}
    d^2_H(\rho_0,\rho_1) = \inf_{(\rho_t,u_t)} \int_0^1 \int_{\R^d} u_t^2 \ud \rho_t \ud t,
\end{equation} 
where $(\rho_t,u_t)$ satisfies the equation 
\[\partial_t \rho_t = -\rho_t u_t.\] 
If measures $\rho_0,\rho_1 \ll \lambda$ for some probability measure $\ud \lambda(x)$, then one can explicitly compute  the minimal cost in \eqref{eqn:ubfisherrao} and obtain
\begin{equation}\label{eqn:ubHellinger}
    d^2_H(\rho_0,\rho_1) = 4\int_{\R^d} \left(\sqrt{\frac{\ud \rho_1}{\ud \lambda}}-\sqrt{\frac{\ud \rho_0}{\ud \lambda}}\right)^2 \ud \lambda.
\end{equation}
Moreover, this expression does not depend on the specific choice of $\lambda$. Indeed, substituting $u_t = -\frac{\partial_t \ud \rho_t/\ud \lambda}{\ud \rho_t/ \ud \lambda}$ into \eqref{eqn:ubfisherrao}, we have
\begin{align*}\int_0^1 \int_{\R^d} u_t^2 \ud \rho_t \ud t & = \int_0^1 \int_{\R^d} \left(\frac{\partial_t \ud \rho_t/\ud \lambda}{\ud \rho_t/\ud \lambda}\right)^2 \ud \rho_t \ud t  = 4 \int_0^1 \int_{\R^d} \left(\partial_t \sqrt{\frac{\ud\rho_t}{\ud \lambda}}\right)^2 \ud \lambda(x) \ud t \\ & \ge 4 \int_{\R^d} \left( \int_0^1\partial_t\sqrt{\frac{\ud\rho_t}{\ud \lambda}}\ud t \right)^2 \ud \lambda(x)= 4\int_{\R^d} \left(\sqrt{\frac{\ud \rho_1}{\ud \lambda}}-\sqrt{\frac{\ud \rho_0}{\ud \lambda}}\right)^2 \ud \lambda.\end{align*} Equality is obtained when 
\begin{equation}\label{eqn:geodesicdH}
    \sqrt{\frac{\ud \rho_t}{\ud \lambda}} = (1-t)\sqrt{\frac{\ud \rho_0}{\ud \lambda}}+t\sqrt{\frac{\ud \rho_1}{\ud \lambda}}.
\end{equation}

\smallskip

From the expression \eqref{eqn:ubHellinger} we can derive immediately that for probability measures $d_H(\rho_0,\rho_1) \le 2\sqrt{2}$.
We note that 
\[ d_H^2(r_0^2 \rho_0, r_1^2 \rho_1) = r_0 r_1 d_H^2(\rho_0,\rho_1) + 4(r_0 -r_1)^2 \] 
and hence $d_H$ is a cone geodesic distance on the space of positive measures satisfying  \cite[(2.1)]{laschos2019geometric}. Thus from Theorem 2.2 and Corollary 2.3 of \cite{laschos2019geometric} (see also \cite[Chapter 2]{ay2017information}) it follows that the spherical Hellinger distance, 
which is obtained by restricting the configurations and paths to probability measures and considering the same path lengths, is given by\footnote{The paper \cite{kondratyev2019spherical} gives an alternative definition of $d_{SH}$ using Benamou-Brenier formulation $\tilde d_{SH}^2(\rho_0,\rho_1) = \inf_{(\rho_t,u_t)} \int_0^1 \int_{\R^d} (u_t-\int \rho_t u_t)^2 \ud \rho_t \ud t$ with geodesic equation $\partial_t \rho_t = -\rho_t (u_t-\int \rho_t u_t)$. As we do not use this formulation, we just remark that showing the equivalence is straightforward. In particular by definition the distances $\tilde d_{SH} \geq d_{SH}$. On the other hand it is direct to check that for the geodesic paths w.r.t $d_{SH}$, identified in  \cite[Theorem 2.7]{laschos2019geometric}, the length w.r.t $\tilde d_{SH}$ is the same as w.r.t $d_{SH}$.
}
\begin{equation} \label{eqn:conicdSH}
d_{SH} (\rho_0,\rho_1) = 2\arccos\Big(1-\frac{d_H^2(\rho_0,\rho_1)}{8}\Big) = 4\arcsin\Big(\frac{d_H(\rho_0,\rho_1)}{4} \Big).
\end{equation}
This implies that $d_{SH} \leq \pi$ and furthermore $d_{SH}(\rho_0, \rho_1) = \pi$ if and only if $\rho_0$ and $\rho_1$ are orthogonal measures.
From the definition \eqref{eqn:conicdSH} one can also observe immediately \begin{equation}\label{eqn:dHdSHequiv}
   d_{SH}(\rho_0,\rho_1) \ge d_H(\rho_0,\rho_1) \, \textrm{ and } \lim_{d_H(\rho_0,\rho_1) \to 0}\frac{d_{SH}(\rho_0,\rho_1)}{d_H(\rho_0,\rho_1)}=1.  
\end{equation}
Furthermore $(\mathcal{P}(\R^d), d_{SH})$ is a geodesic metric space and \cite[Theorem 2.7]{laschos2019geometric} identifies the geodesics, based on geodesics w.r.t $d_H$ in the cone of positive measures. 

\smallskip

The distances $d_H,d_{SH}$ metrize strong convergences of measures, which we present in the lemma below.
\begin{lemma}\label{lem:equivdHL1}
    Suppose $\{\rho_n\}_{n=1}^\infty$ and $\rho$ are measures on $\R^d$ and are all absolutely continuous with respect to some measure $\lambda$. Suppose also that $\rho$ has finite total mass. Then 
    \begin{equation}\label{eqn:equivdHL1}
        \lim_{n\to\infty} d_H(\rho_n,\rho) = 0 \iff \frac{\ud \rho_n}{\ud \lambda} \xrightarrow{L^1(\ud\lambda)} \frac{\ud \rho}{\ud \lambda}.
    \end{equation}
    As a consequence of \eqref{eqn:dHdSHequiv}, if we further assume $\rho_n,\rho$ are probability measures on $\R^d$, then \begin{equation}\label{eqn:equivdSHL1}
        \lim_{n\to\infty} d_{SH}(\rho_n,\rho) = 0 \iff \frac{\ud \rho_n}{\ud \lambda} \xrightarrow{L^1(\ud\lambda)} \frac{\ud \rho}{\ud \lambda}.
    \end{equation}
\end{lemma}
\begin{proof}
    Thanks to \eqref{eqn:dHdSHequiv}, it suffices prove \eqref{eqn:equivdHL1}, which is a direct consequence of the following inequalities (see \cite[Theorem 2.1]{steerneman1983total}): for any $\rho_1, \rho_2 \ll \lambda$,
    \[
d_{H}^2 (\rho_1,\rho_2)  \leq \|\rho_1 - \rho_2\|_{L^1(d\lambda)} \leq \Big(\sqrt{\|\rho_1\|_{L^1(d\lambda)}} + \sqrt{\|\rho_2\|_{L^1(d\lambda)}}\Big) d_{H} (\rho_1,\rho_2).
    \]

\end{proof}

Before presenting the main results of this section, let us state the general assumptions of $\pi$ that we assume throughout this work.
\begin{assumption}\label{ass:genpi}
    The invariant measure $\pi$ and initial condition $\rho_0$ are absolutely continuous with respect to the Lebesgue measure and have density functions $\pi(x),\rho_0(x) $. Let 
    \[\Omega :=\{x\in \R^d, \, \pi(x)>0 \}. \] 
    We require that $\rho_0>0$ in $\Omega$ and $\rho_0=0$ in $\Omega^c$.
\end{assumption}

The pure birth-death dynamics \eqref{eqn:purebd} is
the $d_{SH}$-gradient flow of relative entropy $\KL(\cdot|\pi)$. Under sufficient regularity hypotheses, for any energy functional $\mathcal{G}$, the $d_{SH}$-gradient flow of $\mathcal{G}$ has the form
\begin{equation}\label{eqn:genFRgf}
    \partial_t \rho_t = -\rho_t\left(\frac{\delta \mathcal{G}}{\delta \rho} - \int_{\R^d} \frac{\delta \mathcal{G}}{\delta \rho} \ud \rho_t\right).
\end{equation}
Here $\frac{\delta \mathcal{G}}{\delta \rho}$ is the first variation density of $\mathcal{G}$ at $\rho$ \cite{ambrosio2008gradient, carrillo2006contractions}. 
\smallskip

The following lemma shows the well-posedness of \eqref{eqn:purebd} whenever $\rho_0>0$ on $\Omega$. In addition, the proof reveals the deeper structure of \eqref{eqn:purebd} which indicates exponential convergence  to $\pi$. We prove the convergence rate in Theorem \ref{thm:expconvbdkl}. We note that similar discussion is carried out in the proof of Theorem 2.1 in \cite{liu2022polyak}. 
\begin{lemma}\label{lem:wellposeeps0}
Suppose $\rho_0,\pi$ satisfies Assumption \ref{ass:genpi}, and $\KL(\rho_0|\pi)<\infty$, then there exists a unique solution of \eqref{eqn:purebd}  $\rho \in  C^1\big([0,\infty),L^1(\Omega)\cap\mathcal{P}(\Omega)\big)$, where the differentiability is with respect to $L^1$ norm and which dissipates the $\KL$-divergence. 
\end{lemma}
\begin{proof}
Assume that $\rho \in C^1\big([0,\infty),L^1(\Omega)\cap\mathcal{P}(\Omega)\big)$ is a $\KL$-dissipating solution of \eqref{eqn:purebd}. Then $\rho = 0$  a.e. in space and time outside $\Omega$. In $\Omega$ we have for  $\rho$-a.s. $x$, the function $\eta_t = \log \frac{\rho_t}{\pi}$  satisfies the equation 
\begin{equation}\label{eqn:eqnetat}
    \partial_t \eta_t(x) = -\eta_t(x) + \KL(\rho_t|\pi).
\end{equation}Note that $\KL(\rho_t|\pi)$ depends only on $t$ and is bounded. Thus, using the theory of linear ODEs, should a solution of \eqref{eqn:eqnetat} exist, it has to be of the form
\[    \eta_t(x)  = \eta_0(x) e^{-t}+\psi_t. \] 
Taking exponential, we have 
\begin{equation}\label{eqn:purebdsoln}
    \rho_t(x) = \pi(x) \, \left(\frac{\rho_0(x)}{\pi(x)}\right)^{e^{-t}}\Psi_t.
\end{equation}
where $\Psi_t = e^{\psi_t}$. Since the solution of \eqref{eqn:purebd} must be a probability density for all $t$, we have $\Psi_t^{-1} = \int_{\Omega}  \left(\frac{\rho_0}{\pi}\right)^{e^{-t}} \, d\pi$, which is uniquely determined by $\rho_0$ and $\pi$. Also, $\Psi_t$ must be finite and positive since by H\"older's inequality, $\rho_0^{e^{-t}} \pi^{1-e^{-t}} \in L^1(\Omega)$ and $\int_{\Omega} \rho_0^{e^{-t}} \pi^{1-e^{-t}} \le 1$. This means the equation \eqref{eqn:purebd} has at most one solution. Finally we can verify the existence of a solution by substituting the expression \eqref{eqn:purebdsoln} into \eqref{eqn:purebd}. Direct computation also verifies that $\KL$ divergence is noninceasing. 
\smallskip

\end{proof}

Before stating our result we recall the convergence result of \cite{lu2019accelerating}. 
\begin{theorem} \cite[Theorem 3.3]{lu2019accelerating} \label{thm:oldconv}
Suppose $\rho_0,\pi$ satisfies Assumption \ref{ass:genpi}. Let $\rho_t$ be the solution of \eqref{eqn:purebd} with the initial condition $\rho_0$ satisfying $\KL(\rho_0 | \pi) \leq 1$ and that 
\begin{equation}\label{eq:lowbd}
    \inf_{x \in \Omega} \frac{\rho_0(x)}{\pi(x)} \geq e^{-M}
\end{equation} for some $M>0$. Then 
$$
\KL(\rho_t | \pi) \leq e^{-(2-3\delta)(t-t_\ast)} \KL(\rho_0 | \pi)
$$
for every $\delta \in (0,1/4)$ and all $t\geq t_\ast := \log(M/\delta^3)$.
\end{theorem}
We note that the above theorem requires the condition $\KL(\rho_0 | \pi) \leq 1$; some result would still hold for $\KL(\rho_0 | \pi)< 2$ but not for larger bounds. The result in \cite{liu2022polyak} removes the $\KL(\rho_0|\pi)\le 1$ condition, but they also requires a pointwise upper bound for $\frac{\rho_0}{\pi}$.

\smallskip

We now state our main results that improve the conditions for convergence above by removing the requirement that $\KL(\rho_0 | \pi) \leq 1$ with the only assumption \eqref{eqn:ratiolowerbd}. Furthermore our second result establishes that 
$\KL(\rho_t|\pi)$ contracts exponentially fast to 0 at all times $t \geq 0$.
We remark that asymptotically our rate becomes slower than the one in 
Theorem \ref{thm:oldconv}; once our bounds ensure that $\KL(\rho_t | \pi) \leq 1$ one can apply the results of the above theorem. 

\begin{theorem}\label{thm:expconvbdkl}
Under the assumptions of Lemma \ref{lem:wellposeeps0}, and let $\rho_t$ satisfy the pure birth-death dynamics \eqref{eqn:purebd} with initial condition $\rho_0 \in L^1(\Omega) \cap \mathcal{P}(\Omega)$. Then for any $\rho_0$ satisfying  \begin{equation}\label{eqn:ratiolowerbd}
    \inf_{x\in \Omega} \frac{\rho_0(x)}{\pi(x)} \ge e^{-M},
\end{equation} for some constant $M$, we have for all $t>0$ \begin{equation} \label{eqn:expbdrt1}
    \KL(\rho_t|\pi) \le Me^{-t} + e^{-t+Me^{-t}} \KL(\rho_0|\pi),
\end{equation}
as well as \begin{equation}\label{eqn:conv2ov9}
    \KL(\rho_t|\pi) \le \exp\left(-\int_0^t \lambda(s)\ud s\right) \KL(\rho_0|\pi), \,
 \textrm{ with } \, 
    \lambda(t)= \frac{M^2 e^{-2t}}{9e^{Me^{-t}}(e^{Me^{-t}}-Me^{-t}-1)}.
\end{equation} 
\end{theorem}
\begin{proof}
We first prove \eqref{eqn:expbdrt1}. Recall from the proof of Lemma \ref{lem:wellposeeps0} that 
$$
    \rho_t = \pi \, \left(\frac{\rho_0}{\pi}\right)^{e^{-t}}\Psi_t
$$
for some $\Psi_t >0$. Moreover, 
notice that under the condition \eqref{eqn:ratiolowerbd}, \begin{equation*}
   \frac{1}{\Psi_t} =  \int_{\Omega}\pi\, \left(\frac{\rho_0}{\pi}\right)^{e^{-t}} \ge e^{-Me^{-t}},
\end{equation*} which implies $\Psi_t \le e^{Me^{-t}}$. Therefore \begin{align*}
    \KL(\rho_t|\pi) & = \Psi_t\int_{\Omega} \pi \, \left(\frac{\rho_0}{\pi}\right)^{e^{-t}}(\log \Psi_t + e^{-t} \log \frac{\rho_0}{\pi}) \\ 
 & =\log \Psi_t + \Psi_t e^{-t} \int_{\Omega} \pi\, \left(\frac{\rho_0}{\pi}\right)^{e^{-t}} \log \frac{\rho_0}{\pi} 
 \stepcounter{equation} \tag{\theequation} \label{eqn:medstepPsit} \\ 
 & \le Me^{-t} +  e^{-t+Me^{-t}}\int_{\Omega} \rho_0\log \frac{\rho_0}{\pi} = Me^{-t} +  e^{-t+Me^{-t}}\KL(\rho_0|\pi),
\end{align*}
where in the last inequality we used that if $\rho_0\ge \pi$ then $\left(\frac{\rho_0}{\pi}\right)^{e^{-t} } \le \frac{\rho_0}{\pi}$ and $\log \frac{\rho_0}{\pi}\ge 0$; meanwhile if $\rho_0<\pi$ then $\left(\frac{\rho_0}{\pi}\right)^{e^{-t} } \ge \frac{\rho_0}{\pi}$ and $\log \frac{\rho_0}{\pi}\le 0$.

\smallskip

We now prove \eqref{eqn:conv2ov9}. The proof strategy is a modification of \cite{kondratyev2019spherical}*{Lemmas 2.11 and 2.12} and \cite{kondratyev2020convex}*{Theorem 4.1}. We divide the proof into three steps.

\noindent \emph{Step 1:} Proof of \begin{equation}\label{eqn:entsqr}
    \int_{\Omega} \rho \log \frac{\rho}{\pi}\ud x  \le \frac{e^M-M-1}{M^2}\int_{\Omega} \rho \log^2 \frac{\rho}{\pi}\ud x.
\end{equation}
The key of this step is to prove that for any $r\ge e^{-M}$, \begin{equation}\label{eqn:funcnbd}
    \frac{\log r-1+\frac{1}{r}}{\log^2 r} \le \frac{e^M-M-1}{M^2}.
\end{equation} Let $\varphi(r)=\frac{\log r-1+\frac{1}{r}}{\log^2 r}$, then $\varphi'(r)= \frac{-\log r-r\log r-2+2r}{r^2\log^3 r}$. Now let $\psi(r)=-r\log r-\log r-2+2r$, then $\psi'(r)=1-\log r-\frac{1}{r}\le 0$, so for any $r\ge 1$, $\psi(r) \le \psi(1)= 0$, and therefore $\varphi'(r)\le 0$; on the other hand, when $r\le 1$, $\psi(r)\ge \psi(1)=0$, which again yields $\varphi'(r)\le 0$, and thus when $r\in (e^{-M},\infty)$, the maximum of $\varphi(r)$ is attained at $r=e^{-M}$, which finishes the proof of \eqref{eqn:funcnbd}. Thus taking $r=\frac{\rho}{\pi}$ for any $\rho$ satisfying \eqref{eqn:ratiolowerbd}, we have \begin{equation*}
    \frac{\rho\log \frac{\rho}{\pi}-\rho+\pi}{\rho\log^2 \frac{\rho}{\pi}} \le  \frac{e^M-M-1}{M^2},
\end{equation*}which indicates \eqref{eqn:entsqr} after integration.

\smallskip

\noindent \emph{Step 2:} We strengthen \eqref{eqn:entsqr} into \begin{equation}\label{eqn:funcineq}
    \int_{\Omega} \rho \log \frac{\rho}{\pi}\ud x \le  \frac{9e^M(e^M-M-1)}{M^2}\int_{\Omega} \rho \Big(\log \frac{\rho}{\pi} -\int_{\Omega} \rho\log \frac{\rho}{\pi} \ud x\Big)^2\ud x.
\end{equation}
Let $a=\int \rho \log \frac{\rho}{\pi} \ud x>0$. If $\pb_\pi \left(e^{-M}\le \frac{\rho}{\pi} \le e^\frac{a}{2}\right)\ge \frac{1}{2}$, then, noticing $\log\frac{\rho}{\pi}-a \le -\frac{a}{2}<0$ for $\frac{\rho}{\pi}<e^\frac{a}{2}$, we obtain \begin{align*}
    \int_{\Omega} \rho\left(\log \frac{\rho}{\pi}-a \right)^2 \ud x & \ge e^{-M} \int_{e^{-M}\le \frac{\rho}{\pi}\le e^\frac{a}{2}} \pi \left(\log \frac{\rho}{\pi}-a \right)^2 \ud x \\ & \ge \frac{a^2}{4}e^{-M}  \pb_\pi \left(e^{-M}\le \frac{\rho}{\pi} \le e^\frac{a}{2} \right) \ge \frac{a^2}{8}e^{-M}. \stepcounter{equation} \tag{\theequation} \label{eqn:midregion}
\end{align*}
Otherwise we must have $\pb_\pi \left(\frac{\rho}{\pi}>e^\frac{a}{2} \right)\ge \frac{1}{2}$. Notice for probability densities we have \begin{equation*}
    \int_{\rho>\pi} \rho-\pi \ud x = \int_{\rho<\pi} \pi-\rho\ud x.
\end{equation*} We can estimate the l.h.s. by \begin{equation*}
    \int_{\rho>\pi} \rho-\pi \ud x \ge \int_{\rho> e^\frac{a}{2}\pi} \rho-\pi \ud x \ge (e^\frac{a}{2}-1) \pb_\pi \left(\frac{\rho}{\pi}> e^\frac{a}{2} \right)\ge \frac{1}{2}(e^\frac{a}{2}-1) \ge \frac{a}{4}.
\end{equation*}
On the other hand, \begin{equation*}
     \int_{\rho<\pi} \pi-\rho\ud x \le \int_{\rho<\pi} \pi \log \frac{\pi}{\rho}\ud x,
\end{equation*} which means that 
\begin{equation*}
    \int_{\rho<\pi} \pi \log \frac{\pi}{\rho}\ud x \ge \frac{a}{4}.
\end{equation*}
Hence, using that $\pb_\pi(\rho<\pi)\le \frac{1}{2}$, we obtain
\begin{align*}
    \int_{\Omega} \rho\left(\log \frac{\rho}{\pi}-a \right)^2 \ud x & \ge e^{-M}\int _{\rho<\pi} \pi \left(\log\frac{\pi}{\rho}+a\right)^2 \ud x \\ & \ge e^{-M}\frac{(\int _{\rho<\pi} \pi \log \frac{\pi}{\rho} \ud x)^2}{\pb_\pi(\rho<\pi)} \ge \frac{a^2}{8}e^{-M}. \stepcounter{equation} \tag{\theequation} \label{eqn:largeregion}
\end{align*}
Thus, combining \eqref{eqn:midregion} and \eqref{eqn:largeregion}, in any case we obtain
\begin{equation*}
    \Big( \int_{\Omega} \rho\log\frac{\rho}{\pi} \ud x\Big)^2 \le 8e^M \int_{\Omega} \rho\left( \log \frac{\rho}{\pi}-\int_{\Omega} \rho \log \frac{\rho}{\pi} \ud x \right)^2\ud x.
\end{equation*}
Finally, by further combining  \eqref{eqn:entsqr},  we arrive at
\begin{align*}
    \int_{\Omega} \rho \log \frac{\rho}{\pi}\ud x & \le \frac{e^M-M-1}{M^2}\left(\int_{\Omega} \rho\left( \log \frac{\rho}{\pi}-\int_{\Omega} \rho \log \frac{\rho}{\pi} \ud x \right)^2\ud x+\left(\int_{\Omega} \rho \log \frac{\rho}{\pi} \ud x \right)^2\right) \\ & \le  \frac{(1+8e^M)(e^M-M-1)}{M^2}\int_{\Omega} \rho\left( \log \frac{\rho}{\pi}-\int_{\Omega} \rho \log \frac{\rho}{\pi} \ud x \right)^2\ud x.
\end{align*}

\smallskip

\noindent \emph{Step 3:} Proof of exponential convergence. From the proof of Lemma \ref{lem:wellposeeps0} we have $\Psi_t \ge 1$. By taking infimum on both sides of \eqref{eqn:purebdsoln}, we obtain \begin{equation}\label{eqn:rhotlowerbd}
    \inf_{x\in \Omega} \frac{\rho_t(x)}{\pi(x)}=\Psi_t \inf_{x\in \Omega} \left(\frac{\rho_0(x)}{\pi(x)}\right)^{e^{-t}} \,  \leftstackrel{\eqref{eqn:ratiolowerbd}}{\ge} e^{-Me^{-t}},
\end{equation} In other words, $\rho_t$ satisfies \eqref{eqn:ratiolowerbd} with $Me^{-t}$ playing the role of $M$. Therefore, a combination of direct time differentiation and \eqref{eqn:funcineq} yields \begin{equation*}
    \dfrac{\ud}{\ud t}\KL(\rho_t|\pi)= - \int_{\Omega}\rho_t\left( \log \frac{\rho_t}{\pi}-\int_{\Omega} \rho_t \log \frac{\rho_t}{\pi} \ud x \right)^2\ud x \le -\lambda(t) \KL(\rho_t|\pi).
\end{equation*}
The convergence result \eqref{eqn:conv2ov9} therefore directly follows from a Gronwall inequality. 
\end{proof}

\begin{remark}\label{rmk:improvbdconv}
   The condition \eqref{eqn:ratiolowerbd} can be relaxed or modified if we add conditions on $\pi$. One such modification, inspired by \cite[Proposition 3.23]{chen2023gradient}, is that, suppose for some $p\in [1,\infty)$ we have $M_p(\pi):= \int_{\Omega} |x|^p \ud \pi(x)<\infty$, and in $\Omega$, we have \begin{equation}\label{eqn:newbdassump}
       \frac{\rho_0(x)}{\pi(x)} \ge e^{-M(1+|x|^p)},
   \end{equation} then along \eqref{eqn:purebd}, we have convergence \begin{equation*}
       \KL(\rho_t|\pi) \le Me^{-t}(1+M_p(\pi)) + \exp\big(-t+Me^{-t}(1+M_p(\pi))\big)\KL(\rho_0|\pi).
   \end{equation*}
  The proof follows closely along that of \eqref{eqn:expbdrt1}, with the difference being 
  \begin{align*}
      \frac{1}{\Psi_t} = \int_{\Omega} \left(\frac{\rho_0}{\pi}\right)^{e^{-t}}  \ud \pi &  \ge \int_{\Omega} \left( e^{-M(1+|x|^p)} \right)^{e^{-t}} \ud \pi = \int_{\Omega} e^{-Me^{-t}(1+|x|^p)} \ud \pi \\ & \ge \exp\Big(-\int Me^{-t}(1+|x|^p) \ud \pi   \Big) = \exp\Big(-Me^{-t}(1+M_p(\pi)) \Big),
  \end{align*}
  and we finish the proof after substituting this into \eqref{eqn:medstepPsit}. This new assumption \eqref{eqn:newbdassump} covers almost all reasonable scenarios with $\rho_0$ being Gaussian, as long as $\pi$ has second moment. The upper bound in the assumption of \cite{chen2023gradient} is unnecessary. As is suggested in \cite{lu2019accelerating}, the optimal asymptotic convergence rate should be $e^{-2t}$, which is proved in \cite{domingo2023explicit} under a different set of assumptions.
\end{remark}

\begin{remark}
Combining Langevin dynamics with the birth-death dynamics would result in dynamics with
convergence rate that is at least the maximum of the log-Sobolev constant of $\pi$ and the rates obtained in Theorem \ref{thm:expconvbdkl}. That is, suppose $\pi$ satisfies a logarithmic Sobolev inequality with constant $C_{LSI}$, then for the dynamics \begin{equation*}
    \partial_t \rho_t =  -\rho_t \log \frac{\rho_t}{\pi}+ \rho_t \int_{\Omega} \rho_t \log \frac{\rho_t}{\pi}\ud x + \nabla \cdot \left(\rho_t \nabla \log \frac{\rho_t}{\pi}\right),
\end{equation*} as long as $\rho_0$ satisfies \eqref{eqn:ratiolowerbd}, we have convergence \begin{equation*}
    \KL(\rho_t|\pi) \le \min\left\{\exp\left(-\int_0^t \tilde{\lambda}(s)\ud s\right) \KL(\rho_0|\pi), Me^{-t} + e^{-t+Me^{-t}} \KL(\rho_0|\pi) \right\},
\end{equation*} with \[\tilde{\lambda}(t) = \max\Big\{C_{LSI}, \frac{M^2 e^{-2t}}{9e^{Me^{-t}}(e^{Me^{-t}}-Me^{-t}-1)}\Big\}.\] Convergence rate of $C_{LSI}$ is guaranteed even without the condition \eqref{eqn:ratiolowerbd}.
\end{remark}

At the end of this section, we would like to use two examples to illustrate that the pointwise lower bound condition \eqref{eqn:ratiolowerbd} is not numerically restrictive. In particular, if $V$ has at least quadratic growth at infinity, one can choose $\rho_0$ to be any sufficiently wide round Gaussian to satisfy \eqref{eqn:ratiolowerbd}.

\begin{example} Suppose $V(x)$ is strongly convex and $\frac{m}{2}|x|^2\le V(x)\le \frac{L}{2}|x|^2$. Then we can pick $\rho_0(x)=(\frac{m}{2\pi})^\frac{d}{2}\exp(-\frac{m}{2}|x|^2)$, and therefore \[\inf_{x\in \R^d} \frac{\rho_0}{\pi}=\inf_{x\in \R^d}Z\left(\frac{m}{2\pi}\right)^\frac{d}{2} \exp\left(V(x)-\frac{m}{2}|x|^2\right)\ge Z\left(\frac{m}{2\pi}\right)^\frac{d}{2} \ge \left(\frac{m}{L}\right)^\frac{d}{2},\] which means $\rho_0$ satisfies \eqref{eqn:ratiolowerbd} with $M=\frac{d}{2}\log \frac{L}{m}$. Moreover, after a time of $t\ge \log M=\log (d\log \frac{L}{m})$, the convergence rate becomes $O(1)$.
\end{example}
\begin{example}
Let us consider the double well potential $V(x) = \frac{1}{\epsilon} (1-|x|^2)^2$. We pick $\rho_0 = \left(\frac{1}{2\pi \varepsilon} \right)^\frac{d}{2}e^{-\frac{|x|^2}{2\epsilon}}$, then \[\frac{\rho_0}{\pi}=\frac{Z}{Z_0}\exp\left(\frac{1}{2\varepsilon}(|x|^2-2)^2-\frac{3}{\varepsilon} \right) \ge \left(\frac{1}{2\pi \varepsilon}\right)^\frac{d}{2} Z\exp\left(-\frac{3}{\varepsilon}\right).\] 
Notice that 
\begin{align*}
    Z  = \sigma(\mathbb{S}^{d-1})\int_0^\infty r^{d-1}\exp\Big(-\frac{1}{\varepsilon}(1-r^2)^2\Big) \ud r  \gtrsim \frac{\pi^\frac{d}{2}}{d\sqrt{\pi d}(\frac{d}{2e})^\frac{d}{2}}\exp \left(-\frac{1}{\varepsilon}\right),
\end{align*} which means $\rho_0$ satisfies \eqref{eqn:ratiolowerbd} with $M= \frac{d}{2}\log \frac{d}{2\varepsilon}+\frac{4}{\varepsilon}$, and therefore the burn-in time needed is $O(\log M)=O(\log d+ \log \frac{1}{\varepsilon})$.
\end{example}

\section{Pure birth-death dynamics governed by  chi-squared divergence}
\label{sec:chi}

In this section we consider the spherical Hellinger gradient flow with $\chi^2(\rho|\pi):= \int_{\Omega} \left(\frac{\rho}{\pi}-1\right)^2 \ud \pi$ as the energy functional:
\begin{equation*}
    \partial_t \rho_t =-\rho_t\left( \frac{\rho_t}{\pi}- \int_{\Omega} \frac{\rho_t}{\pi}\ud \rho_t\right).
\end{equation*} This is the dynamics appeared in \cite{lindsey2021ensemble}*{(3.6)}. 
There the authors first derive a related family of dynamics, \cite[(3.1)]{lindsey2021ensemble}, as the continuum limit of the ensemble Monte-Carlo sampling schemes they introduced. For the dynamics \cite[(3.1)]{lindsey2021ensemble}, with kernel $\mathcal{Q}=\textrm{Id}$, they prove exponential decay of the ``reverse'' $\chi^2$-divergence. In the time scaling we consider, this can be stated as  $\chi^2(\pi | \rho_{Z_t}) \lesssim e^{-t}$, where $Z_t$ is the rescaling in time for which  $\frac{\ud Z_t}{\ud t} = \chi^2(\rho_{Z_t} | \pi)+1$. Since $\chi^2(\rho_{Z_t} | \pi) \to 0$ as $t \to \infty$, $\frac{\ud Z_t}{\ud t} \to 1$ as $t \to \infty$. Thus  the exponential rates of \cite[Theorem 1]{lindsey2021ensemble} implies asymptotic exponential rates of order $e^{-t}$ for $\chi^2(\pi | \rho_t)$. However, due to the non-symmetry of the $\chi^2$-divergence, the convergence result on $\chi^2(\pi | \rho_t)$ from \cite{lindsey2021ensemble} does not directly imply a convergence result for $\chi^2(\rho_t 
| \pi)$, although the later is a more natural Lyapunov function for the gradient flow  \eqref{eqn:gfchi2} since it is the underlying energy. In the next theorem, we show that $\chi^2(\rho_t | \pi)$ also contracts exponentially fast and provide a quantitative estimate for the convergence rate.  The authors of \cite{lindsey2021ensemble} also note that the dynamics we consider, \eqref{eqn:gfchi2}, is formally spherical Hellinger gradient flow.

\smallskip

\begin{theorem}\label{thm:expconvbdchi2}
Let $\rho_0,\pi$ satisfy Assumption \ref{ass:genpi}, and let $\rho_t \in C^1\big([0,\infty), L^1(\Omega) \cap \mathcal{P}(\Omega)\big)$ be the solution of \eqref{eqn:gfchi2} with initial condition $\rho_0$.
Then, for any initial probability density $\rho_0(x)$ such that \begin{equation}\label{eqn:lowerbdchi2}
     \inf_{x\in \Omega} \frac{\rho_0(x)}{\pi(x)} \ge e^{-M}
\end{equation} 
for some $M>1$, we have exponential convergence to equilibrium along the dynamics \eqref{eqn:gfchi2}
\begin{equation*}
    \chi^2(\rho_t|\pi) \le \exp\left(-\int_0^t \lambda(s)\ud s\right) \chi^2(\rho_0|\pi),
\end{equation*} with \begin{equation}\label{eqn:chisqrate}
    \lambda(t)=\frac{2}{\big(9+8(e^M-1)e^{-t}\big)\big(1+(e^M-1)e^{-t}\big)}.\end{equation}
\end{theorem}
\begin{proof}
The proof is similar to that of \eqref{eqn:conv2ov9} in Theorem \ref{thm:expconvbdkl}. The core step is the following functional inequality which holds for any $\rho$ satisfying $\inf \frac{\rho}{\pi}\ge e^{-M}$, \begin{equation}\label{eqn:funcineqchi2}
    \int_{\Omega} \frac{\rho^2}{\pi}\ud x -1 \le e^M(1+8e^M)\int_{\Omega} \rho \left(\frac{\rho}{\pi}-\int_{\Omega} \frac{\rho^2}{\pi} \ud x \right)^2 \ud x.
\end{equation}
Let $a=\int_{\Omega} \frac{\rho^2}{\pi} \ud x >1$. If $\pb_\pi \left(e^{-M}\le \frac{\rho}{\pi} \le \frac{a+1}{2}\right) \ge \frac{1}{2}$, then 
\begin{align*}
    \int_{\Omega} \rho \left(\frac{\rho}{\pi}-a \right)^2 \ud x \ge e^{-M}  \int_{e^{-M}\le \frac{\rho}{\pi} \le \frac{a+1}{2}} \pi \left(\frac{\rho}{\pi}-a \right)^2 \ud x & \ge \frac{(a-1)^2}{4e^{M}}
    \pb_\pi \left(e^{-M}\le \frac{\rho}{\pi} \le \frac{a+1}{2} \right)\\ 
    &  \ge \frac{(a-1)^2}{8e^{M}}.
\end{align*}
Otherwise $\pb_\pi \left(\frac{\rho}{\pi}\ge \frac{a+1}{2} \right)\ge \frac{1}{2}$, which means \begin{equation*}
    \int_{\rho<\pi} \pi-\rho\ud x = \int_{\rho\ge \pi} \rho-\pi \ud x \ge \int_{\frac{\rho}{\pi}\ge \frac{a+1}{2}} (\rho-\pi)\ud x \ge \frac{a-1}{2} \pb_\pi \left(\frac{\rho}{\pi}\ge \frac{a+1}{2} \right) \ge \frac{a-1}{4}.
\end{equation*}
Therefore \begin{equation*}
     \int_{\Omega} \rho \left(\frac{\rho}{\pi}-a \right)^2 \ud x \ge e^{-M}  \int_{\rho<\pi} \pi\left(\frac{\rho}{\pi}-a \right)^2 \ud x \ge \frac{(\int_{\rho<\pi} \pi(1-\frac{\rho}{\pi})\ud x)^2}{e^{M}\pb_\pi (\rho<\pi)} \ge \frac{(a-1)^2}{8e^M}.
\end{equation*}
To conclude, \begin{align*}
    \int_{\Omega} \pi \left(\frac{\rho}{\pi}-1 \right)^2 \ud x \le e^M \int_{\Omega} \rho \left(\frac{\rho}{\pi}-1 \right)^2 \ud x & = e^M\left(\int_{\Omega} \rho \left(\frac{\rho}{\pi}-a \right)^2 \ud x + (a-1)^2 \right) \\ & \le e^M(1+8e^M) \int_{\Omega} \rho \left(\frac{\rho}{\pi}-a \right)^2 \ud x.
\end{align*}
This finishes the proof of \eqref{eqn:funcineqchi2}. Now let us return to the dynamics \eqref{eqn:gfchi2}. Taking time derivative, we have for $e^{-M(t)}=\inf\frac{\rho_t}{\pi}$,
\begin{align*}
    \dfrac{\ud}{\ud t}\int_{\Omega} \pi \left(\frac{\rho_t}{\pi}-1 \right)^2 \ud x  & = - 2\int_{\Omega} \rho_t \left(\frac{\rho_t}{\pi}-\int_{\Omega} \frac{\rho_t^2}{\pi} \ud x \right)^2 \ud x \\
    & \leftstackrel{\eqref{eqn:funcineqchi2}}{\le} -\frac{2}{e^{M(t)}(1+8e^{M(t)})}\int_{\Omega} \pi\left(\frac{\rho_t}{\pi}-1 \right)^2 \ud x.
\end{align*}
 By Gronwall inequality, this means the dynamics converge exponentially with instantaneous rate $\lambda(t) = \frac{2}{e^{M(t)}(1+8e^{M(t)})}$. Finally, notice that \begin{equation*}
    \dfrac{\ud}{\ud t} \frac{\rho_t}{\pi} = - \frac{\rho_t^2}{\pi^2} + \frac{\rho_t}{\pi} \int_{\Omega} \frac{\rho_t^2}{\pi} \ud x \ge - \frac{\rho_t^2}{\pi^2} + \frac{\rho_t}{\pi}.
\end{equation*}
Therefore, solving the ODE,  one has \begin{equation*}
   e^{-M(t)} \ge \frac{e^t}{e^M+e^t-1},
\end{equation*}which gives the convergence rate in \eqref{eqn:chisqrate}.
\end{proof}
Notice that one has $\lim_{t\to\infty}\lambda(t) = \frac{2}{9}$, which we believe is suboptimal based on the results in Theorem \ref{thm:expconvbdkl} (i) as well as the observation made in \cite{lindsey2021ensemble}. On the other hand, if $M\gg 1$, then the instantaneous convergence rate is $O(1)$ only when $e^{M-t}= O(1)$, which means $t\ge O(M)$. Hence the waiting time of \eqref{eqn:gfchi2} is longer than that of \eqref{eqn:purebd}, which is $O(\log M)$. 

\section{Kernelized dynamics and its particle approximations}

In this section we investigate a particle-based approximation to the dynamics \eqref{eqn:purebd}. We first introduce a nonlocal approximation of \eqref{eqn:purebd} that is based on regularizing the relative entropy:
 \begin{equation}\label{eqn:gfkerbd}\tag{$\text{BD}_\varepsilon$}
 \!   \partial_t \rhove = -\rhove\left[\log \left(\frac{\Kve*\rhove}{\pi}\right) + \Kve * \left( \frac{ \rhove}{\Kve* \rhove} \right) - \! \int  \! \log \left(\frac{\Kve*\rhove}{\pi}\right) d \rhove -1 \right].
\end{equation}
It is the spherical Hellinger gradient flow of the regularized entropy 
\begin{equation}\label{eqn:funckerkl}
    \fve(\rho) = \int  \log (\Kve * \rho) d \rho - \int \log \pi d \rho = \int \log(\Kve * \rho) d \rho  + \int V d \rho + C, 
\end{equation}
where $C = \log \left( \int \exp(-V(x))dx \right)$.
We first study the energy  $\fve$ on the whole space. In Sections 
\ref{sec:wpBDeps} and \ref{sec:gamma} we study the well posedness and the convergence as $\varepsilon \to 0$ of the gradient flow on the torus. 

We now state the conditions on the kernel $\Kve$ that we require for our results. 
\begin{assumption}\label{assump:kernel}
The  kernel $\Kve(x-y)$ is of the form $\Kve(x-y)= \varepsilon^{-d}K(\frac{x-y}{\varepsilon})$, 
where $K \in C^\infty(\R^d) \cap L^\infty(\R^d)$ satisfies the following:
 \begin{listi}
     \item $\int_{\R^d} K \ud x=1$ and $K$ is positive definite, in the sense that for any function $f\in C_c^\infty(\R^d)$, 
     \[\int K(x-y) f(x) f(y) \ud x \ud y \ge 0.\]
     \item $K$ is radially symmetric, i.e. $K(x) = \calK(|x|)$, and $M_4(K):=\int_{\R^d} |x|^4 K (x) \ud x <\infty$.
 \end{listi}
\end{assumption}
\noindent Assumption \ref{assump:kernel} also indicates that there exists a kernel $\xi\ge 0$ such that $K = \xi*\xi$ and $\int_{\R^d} \xi=1$, namely $\hat \xi = \sqrt{\hat K}$.  One example that satisfies Assumption \ref{assump:kernel} is the Gaussian kernel $K(x) = (2\pi)^{-\frac{d}{2}}e^{-\frac{|x|^2}{2}}$, in which case $\xi = K_{1/\sqrt{2}}$. We also use $\xi_\varepsilon$ to denote $\varepsilon^{-d} \xi(\frac{\cdot}{\varepsilon})$.

\subsection{Quantitative distance between minimizers} \label{sec:minimizer-approx}
The $\Gamma$-convergence of $\fve$, defined in \eqref{eqn:funckerkl}, to relative entropy $\KL(\rho|\pi)$ and the convergence of minimizers are already proved in \cite{carrillo2019blob}, which we restate in the following Theorem \ref{thm:blobengyconv}. \begin{theorem}[\cite{carrillo2019blob}] \label{thm:blobengyconv}
Suppose $\pi = \exp(-V) \in \mathcal{P}_2(\R^d)$. Let $\Kve$ satisfy Assumption \ref{assump:kernel}. 
\begin{listi}
    \item (\cite{carrillo2019blob}*{Theorem 4.1}) As $\varepsilon\to 0$, $\fve$ defined in \eqref{eqn:funckerkl} $\Gamma$-converges to $\calF(\rho):= \KL(\rho|\pi)$, in the sense that for any sequence $\rhove \wkstc \rho$, we have $\liminf_{\varepsilon\to 0} \fve (\rhove) \ge \calF(\rho)$. Moreover, $\limsup_{\varepsilon\to 0} \fve(\rho) \le \calF(\rho)$.
    \item (\cite{carrillo2019blob}*{Theorem 4.5}) Suppose in addition that $K$ is Gaussian, and that there exists a constant $C$ such that $V(x) \ge C(|x|^2-1)$, then minimizers of $\fve$ over $\mathcal{P}_2(\R^d)$ exist. Moreover, for any sequence $(\pi_\varepsilon)_\varepsilon$ such that $\pi_\varepsilon \in \mathcal{P}_2(\R^d)$ is a minimizer of $\fve$, we have, up to a subsequence, $\pi_\varepsilon \wkstc \pi$.
\end{listi}
\end{theorem} However, \cite{carrillo2019blob} did not prove quantitatively how close the minimizers $\pi_\varepsilon$ are to $\pi$. In the case $V$ is sufficiently regular and has between quadratic and quartic growth, we can prove a more quantitative result. We would like to comment here that, while our assumption on $V$ is slightly stronger than that of Theorem \ref{thm:blobengyconv}, we do not require $K$ to be Gaussian in our following theorem.

\begin{theorem}\label{thm:quanminconv}
Suppose $\Kve$ satisfies Assumption \ref{assump:kernel} with some $\xi \in \mathcal{P}_4(\R^d)$. Suppose $\pi$ satisfies a Talagrand inequality \cite{otto2000generalization} with some constant $m>0$: \begin{equation} \label{eqn:talagrand} W_2(\rho,\pi) \le \sqrt{\frac{2}{m}\KL(\rho|\pi)}\end{equation} for any probability measure $\rho$. Suppose also that  \begin{equation} \label{eqn:thm5condV} \|D^2 V(x)\|\le L(1+|x|^2)  \, \textrm{ and } \|D^4 V(x)\|\le L\end{equation} for some $L>1$. Then, for any $0<\varepsilon < \sqrt{\frac{m}{4LM_2(\xi)}}$, let $\pi_\varepsilon$ be a minimizer of $\fve$ defined in \eqref{eqn:funckerkl}, we have for some constant $C=C(\pi,\xi)$ independent of $\varepsilon$,
\begin{equation}\label{eqn:w2piepspi}
  W_2(\pi_\varepsilon,\pi) \le C \varepsilon, 
\end{equation}and \begin{equation}\label{eqn:minvalfve}
    0\ge \fve(\pi_\varepsilon) \ge -C\varepsilon^2.  
\end{equation}\end{theorem}
We remark that analogous result holds if we consider the energy on the torus $\T^d$. The integrals considered are on the torus, while for evaluating convolutions, all of the functions are extended periodically to $\R^d$.

\begin{proof}
To get started, we have for arbitrary $\rho$, \begin{equation}\label{eqn:fvesmallerKL}\fve(\rho)-\KL(\rho|\pi) = -\KL(\rho|\Kve*\rho) \le 0.\end{equation} In particular, taking $\rho=\pi_\varepsilon$ where $\pi_\varepsilon$ is any minimizer of $\fve$, we have \begin{equation}\label{eqn:fvepieple0} \fve(\pi_\varepsilon)\le \fve(\pi) \le \KL(\pi|\pi)=0.\end{equation} On the other hand, since $\Kve=\xieps* \xieps$ and $x\mapsto \log x$ is concave, we use Jensen inequality to obtain
\begin{equation}\label{eqn:jensen}
    \log(\Kve*\rho) = \log(\xieps*\xieps*\rho) \ge \xieps*\log(\xieps*\rho)
\end{equation}
We then use the regularity of $V$ and Talagrand inequality:
\begin{align*}
    \fve(\rho) & = \int \rho \log(\Kve * \rho) - \int \rho\log \pi \\ & \leftstackrel{\eqref{eqn:jensen}}{\ge} \int (\xieps * \rho) \log(\xieps * \rho) - \int \rho\log \pi \\ & = \KL(\xieps*\rho | \pi) + \int (\xieps*\rho-\rho) \log \pi \\ & \ge   \frac{m}{2}W_2^2(\xieps*\rho,\pi) + \int \rho(x) \int \xieps(x-y) (V(x) -V(y)) \ud y \ud x. \stepcounter{equation} \tag{\theequation} \label{eqn:fverho1}
\end{align*} For the first term, we can use triangle inequality \begin{equation} \label{eqn:fverho2} W_2(\xieps*\rho,\pi) \ge |W_2(\rho,\pi)- W_2(\xieps*\rho,\rho)| \ge W_2(\rho,\pi)-\sqrt{M_2(\xi)}\varepsilon.\end{equation} For the second term, by Taylor expansion, we have\footnote{Here we use the short hand notation $(y-x)^3 : D^3 V(x) := \sum_{i,j,k} (y_i-x_i)(y_j-x_j)(y_k-x_k) \partial_{ijk}V(x)$, similarly for the fourth derivative term.} \[V(y)-V(x) = \nabla V(x) \cdot (y-x) + \frac{1}{2}(y-x)^\top D^2V(x) (y-x) + \frac{1}{6}(y-x)^3 : D^3 V(x) + \frac{1}{24}(y-x)^4 : D^4 V(\zeta)  \] for some $\zeta \in [x,y]$. Then, we appeal to the symmetry of the kernel $\xi$ and derive,  \begin{align*}
    \int \rho(x)& \int  \xieps(x-y) (V(x) -V(y)) \ud y \ud x \\ & =  \int \rho(x) \int \xieps(x-y) \left(\frac{1}{2}(y-x)^\top D^2 V(x) (y-x) + \frac{1}{24}(y-x)^4 : D^4 V(\zeta)\right) \ud y \ud x \\ & \leftstackrel{\eqref{eqn:thm5condV}}{\ge} -\frac{L}{2}\iint(1+|x|^2) \xieps(x-y) |y-x|^2 \rho(x)   - \frac{L}{24} \iint \rho(x) \xieps(x-y) |y-x|^4 \\& = -\frac{L\varepsilon^2}{2}M_2(\xi) \int \rho(x) (1+|x|^2) \ud x - \frac{L\varepsilon^4}{24}M_4(\xi) \\ &  \ge  -L\varepsilon^2M_2(\xi)\Big(M_2(\pi)+W_2^2(\rho,\pi)+\frac{1}{2}\Big)- \frac{L\varepsilon^4}{24}M_4(\xi). \stepcounter{equation} \tag{\theequation} \label{eqn:fverho3}
\end{align*}  Here in the last step, we use that for any $\gamma(x,y)$ being a coupling between $\rho$ and $\pi$, \[\int|x|^2 \ud \rho \le 2\inf_\gamma \int|x-y|^2 \ud \gamma(x,y) + 2\int |y|^2 \ud \pi(y)= 2 W_2^2(\rho,\pi)+2M_2(\pi). \] If $W_2(\pi_\varepsilon,\pi) \le \varepsilon \sqrt{M_2(\xi)}$ then the theorem is immediate. Otherwise, we combine \eqref{eqn:fverho1}, \eqref{eqn:fverho2} and \eqref{eqn:fverho3} for $\rho=\pi_\varepsilon$ to obtain that there exists a constant $C=C(\pi,\xi)>0$ independent of $\varepsilon$ such that  \begin{align*}
    0 \ge \fve(\pi_\varepsilon) \ge \frac{m}{2}(W_2(\pi_\varepsilon,\pi) -\sqrt{M_2(\xi)}\varepsilon)^2 - L\varepsilon^2 W_2^2(\pi_\varepsilon,\pi)M_2(\xi)-C\varepsilon^2,
\end{align*} which leads to \eqref{eqn:w2piepspi} after completing the square, and \eqref{eqn:minvalfve} after optimizing the right hand side above with $W_2(\pi_\varepsilon,\pi)$. \end{proof}

\begin{remark}We are not able to prove the sharpness of the error estimate \eqref{eqn:w2piepspi}. However, we demonstrate via numerical experiments that the error bound $O(\varepsilon)$ above seems to be optimal. Consider the target measure $\pi(x) \propto \, e^{-V(x)}$ on the 1-dimensional torus $\T = \R/2\pi\mathbb{Z}$ with potential $V(x) = \sin x + 2\sin 2x$. We solve the equations \eqref{eqn:purebd} and \eqref{eqn:gfkerbd} using the finite difference method and use numerical integration to compute the evolution of $\KL(\rhove_t|\pi)$ for various values of $\varepsilon$. The left side of Figure \ref{fig:1dtorus} shows that for a fixed $\varepsilon > 0$, $\KL(\rhove_t|\pi)$ approaches to a positive constant as $t\to\infty$. The right side of   Figure \ref{fig:1dtorus} plots the function $\varepsilon \mapsto \KL(\rhove_T|\pi)$ for $T=15$, which indicates that $\KL(\rhove_\infty | \pi) \approx O(\varepsilon^2)$ as $\varepsilon \rightarrow 0$. This is consistent with the scaling in \eqref{eqn:w2piepspi} in view of the Talagrand inequality \eqref{eqn:talagrand}. 
\end{remark}

\begin{figure} 
    \begin{minipage}{3in}
    \includegraphics[width=2.85in]{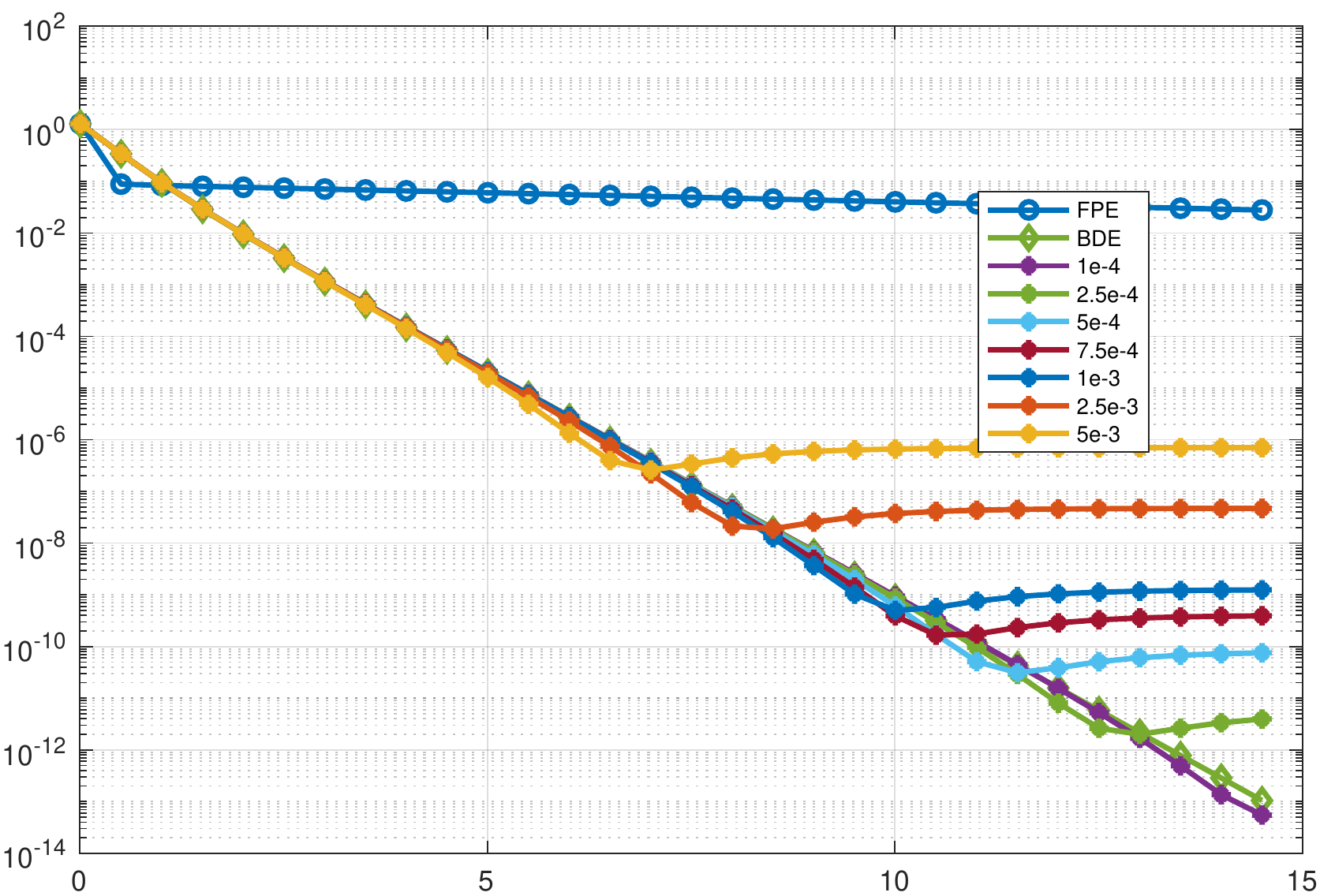}
    \end{minipage}
    \begin{minipage}{3in}
    \includegraphics[width=2.85in]{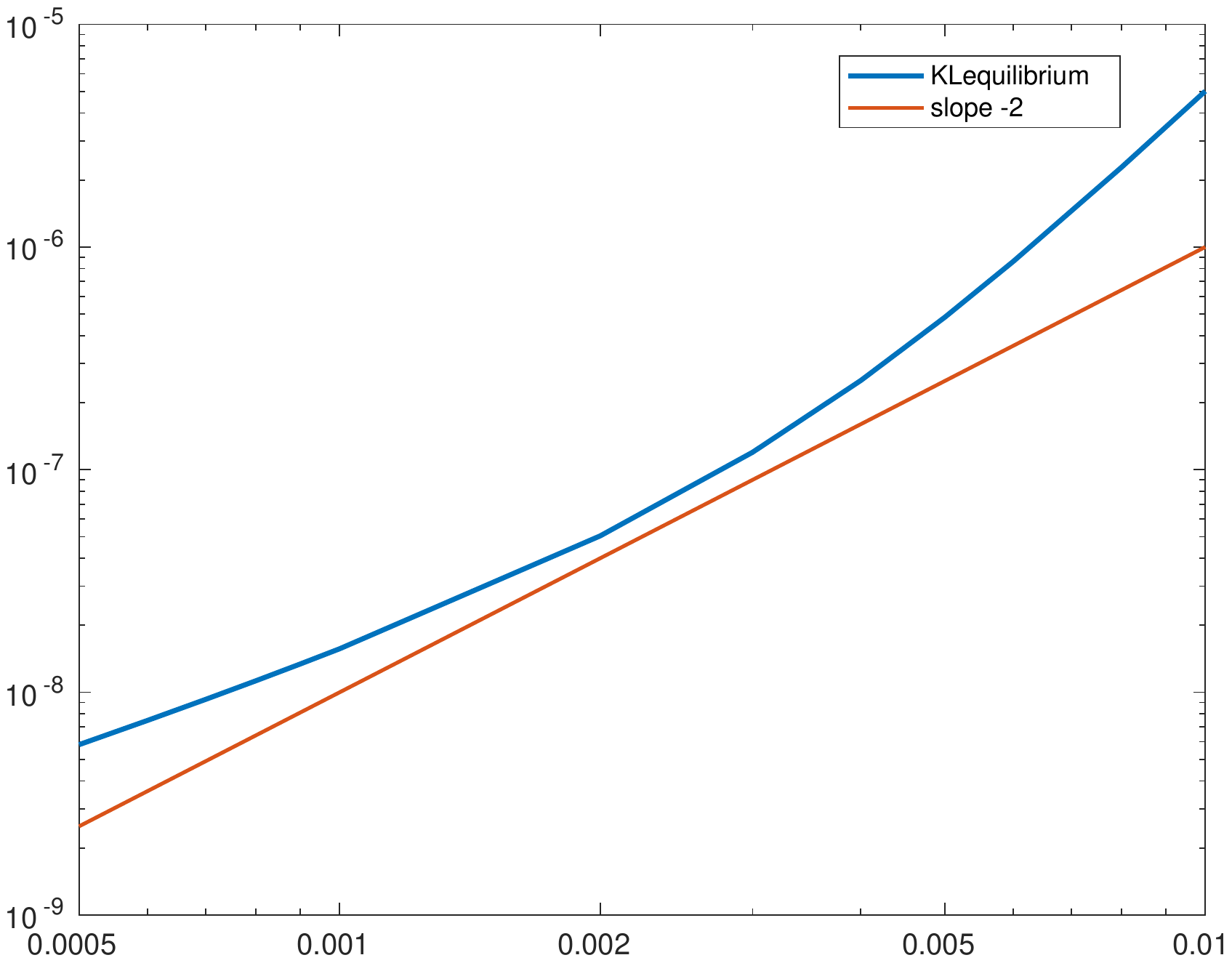}
    \end{minipage}
    \caption{1D torus example. Left: evolution of $\KL(\rhove_t | \pi)$ for various $\varepsilon$, which heuristically goes to some fixed number as $t\to\infty$ for every fixed $\varepsilon$. Right: the relationship between $\varepsilon$ and $\KL(\rhove_\infty|\pi)$, which scales like $O(\varepsilon^2)$ as $\varepsilon\to 0$.}
    \label{fig:1dtorus}
\end{figure}

\subsection{Well-posedness of gradient flows \eqref{eqn:gfkerbd}}\label{sec:wpBDeps}


In this section we develop the well-posedness of \eqref{eqn:gfkerbd}  considered on the torus $\T^d$ and establish the elements of its gradient flow structure relevant for the convergence argument. Due to the smoothness of the terms of the equation it is easy to show well-posedness in the classical sense of ODE in Banach spaces. More precisely we first show local well-posedness in the spaces of measures on the torus which are bounded from above and below by a positive constant. We then show $L^\infty$ bounds that allow us to extend the solution up to some large $T_\varepsilon \xrightarrow{\varepsilon\to 0} \infty$.
After establishing the existence and uniqueness of classical $L^1$ solutions, we show that both energies $\calF$ and $\fve$ are $\lambda$-geodesically convex with respect to $d_{SH}$, in the restricted sense described in Definition \ref{def:convex}. We then characterize the subdifferential of  $\fve$ with respect to the $d_{SH}$ geometry. These allow us to show that the classical solutions coincide with gradient flow solutions, as well as the curves of maximal slope. 

\smallskip

For $C>1$ let 
\begin{equation} \label{eq:PC}
    \mathcal{P}_C =\left\{\rho \in \mathcal{P}(\mathbb{T}^d)\cap L^1(\mathbb{T}^d), \text{ and } \frac{1}{C} \le \rho(x) \le C, \text{ a.e. in }\mathbb{T}^d \right\}.
\end{equation}

\begin{lemma} \label{lem:loc-existBDe}
Assume $\Kve$ satisfies Assumption \ref{assump:kernel}, and $\pi(x) \propto \, e^{-V(x)}$ where $V$ is a $C^2$ function of $\T^d$. 
For any $C>1$ and  $\rho_0 \in \mathcal{P}_{C/2}$ 
there exists $T>0$, independent of $\rho_0$ and a unique $\rhove \in C^1([0,T], (\mathcal{P}_C, \|\, \cdot \, \|_{L^1}))$ solving \eqref{eqn:gfkerbd}. 
Moreover $\rhove \in C^1([0,T], (\mathcal{P}_C, \|\, \cdot \, \|_{L^2}))$.
Finally if $\rho_0 \in C^1(\T^d)$ then $\rhove_t \in C^1(\T^d)$ for all $t \in [0,T]$. 
\end{lemma}
\begin{proof}
The existence and uniqueness follow from the classical existence of ODE in Banach spaces. In particular we claim that te right-hand side of \eqref{eqn:gfkerbd},
\[ A(\rhove) := \rhove\left[\log \left(\frac{\Kve*\rhove}{\pi}\right) + \Kve * \left( \frac{ \rhove}{\Kve* \rhove} \right) - \! \int \log \left(\frac{\Kve*\rhove}{\pi}\right) d \rhove_t -1 \right] \]
is Lipschitz continuous in $L^1$ norm on $\mathcal{P}_C$ and is uniformly bounded in $L^\infty$. That is, there exist $L,M \in (0, \infty)$ such that for all $\rho, \sigma \in \mathcal{P}_C$,
\[ \| A(\rho) - A(\sigma) \|_{L^1(\T^d)} \leq L \| \rho -  \sigma \|_{L^1(\T^d)} \quad \textrm{ and }\quad \| A(\rho)\|_{L^\infty(\T^d)} \leq M. \]
Showing this is straightforward and we only sketch the argument. Observe that $\rho \mapsto \Kve * \rho$ is a 1-Lipschitz mapping on $L^1$ and that $\frac{1}{C} < \Kve * \rho < C$ for all $\rho \in \mathcal{P}_C$. We also use  that $V$ is bounded and continuous on $\T^d$. Furthermore note that 
\begin{align*}
\left \| \Kve* \left(\frac{\sigma}{\Kve*\sigma} - \frac{\rho}{\Kve*\rho} \right)  \right \|_{L^1} 
\leq \left \| \frac{\sigma}{\Kve*\sigma} - \frac{\rho}{\Kve*\rho}  \right \|_{L^1} 
&  \leq  \left\|  \frac{\sigma - \rho}{\Kve*\sigma}  \right\|_{L^1}  +  \left\| \frac{\rho \, \Kve*(\sigma -\rho)|}{\Kve*\sigma \, \Kve*\rho} \right\|_{L^1}  \\
& \leq (C + C^3) \| \rho -  \sigma \|_{L^1}. 
\end{align*} 
Combining these facts provides the desired constants $L$ and $M$. 
Thus there exists a $T>0$ and $\rhove \in C^1([0,T], L^1(\T^d))$ such that $\rhove_t \in \mathcal{P}_C$ for all $t \in [0,T]$, $\rhove(0) = \rho_0$, and $\partial_t \rhove_t = A(\rhove_t)$ for all $t \in [0,T]$, where the derivative is taken in $L^1(\T^d)$. The fact that $\rhove \in C^1([0,T], (\mathcal{P}_C, \|\, \cdot \, \|_{L^2}))$ follows from $\rhove \in C^1([0,T], (\mathcal{P}_C, \|\, \cdot \, \|_{L^2}))$ and the $L^\infty$ boundedness of $\rhove_t$ and $A(\rhove_t)$.

\smallskip

The proof  that the solution is in $C^1(\T^d)$ follows in the standard way. Namely we first observe that, once we know that $\rhove$ is a solution in $\mathcal P_C$ then it satisfies
$\partial_t \rhove =  h(x,t) \rhove$ where $h$ is continuous and bounded. This implies that $\rho(t) \in C(\T^d)$ for all $t \in [0,T]$. To show the $C^1$ regularity one needs to show that $u= \partial_{x_i} \rhove$ satisfies an integral equation. Note that  taking a derivative of 
\eqref{eqn:gfkerbd} gives that $u$ satisfies a linear integral equation. Obtaining the existence of the solution and showing that it is the derivative of $\rhove$ is straightforward and we do not present the details to conserve space. 
\end{proof}
%
\nc
The next few lemmas aim to establish $L^\infty$ upper and lower bounds for $\rhove_t$, which allows us to extend the local existence theory to long time intervals. 
%

\begin{lemma}\label{lem:rhokverhoclose}
Suppose that $K \in C_c^\infty(B_1)$ satisfies Assumption \ref{assump:kernel}, and that $w=\log \rho$ is $L$-Lipschitz continuous, then for any fixed $x$, we have \begin{equation}\label{eqn:comparerhokverho}
    e^{-L\varepsilon} \le\frac{\Kve*\rho(x)}{\rho(x)} \le e^{L\varepsilon}.
\end{equation}
\end{lemma}
\begin{proof}
\begin{align*}
    \frac{\Kve*\rho(x)}{\rho(x)} & = \int_{|y-x|\le \varepsilon} \Kve(x-y) \exp\left( w(y)-w(x)\right)\ud y \\ & \le  \int_{|y-x|\le \varepsilon}\Kve(x-y)e^{L\varepsilon}\ud y = e^{L\varepsilon}.
\end{align*}
Similarly, we have $\frac{\Kve*\rho(x)}{\rho(x)} \ge e^{-L\varepsilon}$.
\end{proof}
\nc
\begin{lemma}\label{lem:upperbdgradw} 
Let $\varepsilon>0$. Assume $\rhove_t\in C^1([0,T],\mathcal{P}_C(\mathbb{T}^d))$ is a solution of \eqref{eqn:gfkerbd} with initial condition $\rho_0 \in C^1(\mathbb{T}^d) \cap \mathcal{P}_C(\T^d)$ for some $C>1$, and let $\wvet = \log \rhove_t$. Then there exists  constant $\overline C$ such that for any $t\in [0, \min\{T, T_\varepsilon\}]$, 
\begin{equation}\label{eqn:derivbdwvet}
    \|\nabla \wvet \|_{L^\infty(\T^d)} \le e^{(1+2e^{C_0})t} \left( \|\nabla w_0 \|_{L^\infty(\T^d)}+ \frac{\|\nabla V\|_{L^\infty(\T^d)}}{1+2e^{C_0}} \right) - \frac{\|\nabla V\|_{L^\infty(\T^d)}}{1+2e^{C_0}},
\end{equation}
where  $T_\varepsilon$ is defied in the proof and satisfies that $T_\varepsilon \to \infty$ as $\varepsilon \to 0$.
\end{lemma}
\begin{proof}
We start with the observation that $\wvet$ satisfies the equation \begin{equation*}
    \partial_t \wvet = - \log \frac{\Kve* \rhove_t}{\pi}-\Kve* \frac{\rhove_t}{\Kve* \rhove_t} + \fve(\rhove_t) +1.
\end{equation*}
Taking spatial derivative, we have 
\begin{align*}\partial_t \partial_i \wvet & = - \frac{ \Kve* \partial_i\rhove_t}{\Kve* \rhove_t} -\partial_i V - \Kve*\partial_i  \frac{\rhove_t}{\Kve* \rhove_t}\\ &  = - \frac{\Kve* (\rhove_t \partial_i \wvet)}{\Kve* \rhove_t} - \partial_i V - \Kve* \frac{\rhove_t \partial_i \wvet \Kve* \rhove_t - \rhove_t \Kve * \rhove_t \partial_i w_t^{(\varepsilon)}}{(\Kve* \rhove_t)^2}. \stepcounter{equation} \tag{\theequation} \label{eqn:dtwvet} \end{align*}
Now fix an index $i$, and let $L_\varepsilon(t) = \sup_{x\in \T^d} |\partial_i\wvet|$, we have 
\begin{align*}
    |\partial_t \partial_i \wvet| \leftstackrel{\eqref{eqn:dtwvet}}{\le} |\partial_i V| + \left(1+2\Kve* \frac{\rhove_t}{\Kve* \rhove_t} \right) L_\varepsilon(t) \leftstackrel{\eqref{eqn:comparerhokverho}}{\le} \|\nabla V\|_{L^\infty(\T^d)} +(1+2e^{\varepsilon L_\varepsilon(t) })L_\varepsilon(t).
\end{align*}
After taking supremum on the left  side above, this turns into \begin{equation}\label{eqn:bootstrapingLvepst}
    L_\varepsilon'(t) \le \|\nabla V\|_{L^\infty(\T^d)} +(1+2e^{\varepsilon L_\varepsilon(t) })L_\varepsilon(t).
\end{equation}
Let $\tilde{L}_\varepsilon = \varepsilon L_\varepsilon$, then 
\[
\frac{\ud}{\ud t} \tilde{L}_\varepsilon(t) \le \varepsilon \|\nabla V\|_{L^\infty(\T^d)} + \tilde{L}_\varepsilon(t)(1+2e^{\tilde{L}_\varepsilon(t)}).
\]
Notice that for every $\varepsilon>0$, the solution $z_\varepsilon$ to the ODE problem
$$
\frac{\ud}{\ud t} z_\varepsilon(t) = \varepsilon \|\nabla V\|_{L^\infty(\T^d)} +z_\varepsilon(t)(1+2e^{z_\varepsilon(t)}), \quad z_\varepsilon(0)= \varepsilon L(0)
$$
blows up at the finite time 
$$
\tau_{\varepsilon} := \int_{z_\varepsilon(0)}^\infty \frac{1}{\varepsilon \|\nabla V\|_{L^\infty} + z(1+2 e^z)} \ud z < \infty.
$$
However, as $\varepsilon \rightarrow 0$,  one has that  $z_\varepsilon(0) \rightarrow 0$ and $\varepsilon \|\nabla V\|_{L^\infty(\T^d)}  \rightarrow 0$. Therefore, 
$$
\tau_\varepsilon \rightarrow \int_0^\infty \frac{1}{z (1+2 e^z)} \ud z = \infty.
$$
Notice that $z_\varepsilon(t)$ is an increasing function of both $t$ and $\varepsilon$. Let $T_\varepsilon = \frac12 \tau_\varepsilon$. 
Let $\overline C =z_{\varepsilon}(T_\varepsilon)$. Then 
$\tilde{L}_\varepsilon(t) \leq \overline C $
for all $t \leq \min\{T, T_\varepsilon\}$.
Taking account of above in the differential inequality \eqref{eqn:bootstrapingLvepst}, one obtains from Gronwall's inequality that 
for all $t\in [0, \min\{T, T_\varepsilon\}]$,
\[
L(t) \leq e^{(1+2e^{\overline C})t} \left(L(0) + \frac{\|\nabla V\|_{L^\infty(\T^d)}}{1+2e^{\overline C}}\right) - \frac{\|\nabla V\|_{L^\infty(\T^d)}}{1+2e^{\overline C}},
\]
which is precisely \eqref{eqn:derivbdwvet}.
\end{proof}

\begin{lemma}\label{lem:upperbdrhovet}
Under the same conditions as Lemma \ref{lem:upperbdgradw}, suppose that for all $t\in[0,T]$, $\wvet=\log \rhove_t$ is $L$-Lipschitz continuous. Then 
\begin{equation}\label{eqn:supbdrve}
    \sup_{\T^d} \frac{\rhove_t}{\pi} \le \exp\left((1-e^{-t})\left(L\varepsilon + \KL(\rho_0|\pi) +1\right) + e^{-t} \log \frac{\rho_0}{\pi}\right)
\end{equation}and 
\begin{equation}\label{eqn:infbdrve}
    \inf_{\T^d} \frac{\rhove_t}{\pi} \ge \exp\left(-(1-e^{-t})\left(L\varepsilon + e^{L\varepsilon} +C_\pi\varepsilon^2\right) + e^{-t} \log \frac{\rho_0}{\pi}\right).
\end{equation}
Here $C_\pi$ is the constant depending on $\pi$ that appears in Theorem \ref{thm:quanminconv}, such that $\fve(\pi_\varepsilon) \ge -C_\pi \varepsilon^2$.
\end{lemma}
To applying this lemma one can use the Lipschitz estimate of Lemma \ref{lem:upperbdgradw} and  take $L $ to be the right hand side of  
\eqref{eqn:derivbdwvet} for $t=T_\varepsilon$. 
\begin{proof}
Notice that \begin{align*}
    \partial_t \log \frac{\rhove_t}{\pi} & = -\log \frac{\Kve* \rhove_t}{\pi}- \Kve* \frac{\rhove_t}{\Kve*\rhove_t} + \fve(\rhove_t)+1 \\ & \leftstackrel{\eqref{eqn:comparerhokverho}, \eqref{eqn:fvesmallerKL}}{\le} -\log \frac{\rhove_t}{\pi}+ L\varepsilon+\KL(\rho_0|\pi)+1.
\end{align*}Using Gronwall's inequality, we have \begin{equation*}
    \log \frac{\rhove_t}{\pi} \le (1-e^{-t})\left(L\varepsilon + \KL(\rho_0|\pi) +1\right) + e^{-t} \log \frac{\rho_0}{\pi}.
\end{equation*}
The proof of lower bound is similar. Since \begin{align*}
    \partial_t \log \frac{\rhove_t}{\pi}  \leftstackrel{\eqref{eqn:comparerhokverho}}{\ge} -\log \frac{\rhove_t}{\pi}- L\varepsilon-e^{L\varepsilon}-C_\pi \varepsilon^2,
\end{align*}we obtain \begin{equation*}
    \log \frac{\rhove_t}{\pi} \ge e^{-t}\log \frac{\rho_0}{\pi} -(1-e^{-t})(L\varepsilon +e^{L\varepsilon} +C_\pi\varepsilon^2).
\end{equation*}
\end{proof}

\begin{theorem}\label{thm:wlBDeps}
    Let $\Kve$ satisfy Assumption \ref{assump:kernel} and is supported in $B(0,1)$. Let $\pi(x) \propto \, e^{-V(x)}$ where $V$ is a $C^2$ function of $\T^d$. Let  $C>1$ be such that $\pi \in C^1(\T^d)\cap \mathcal{P}_C$.
Consider the solution of \eqref{eqn:gfkerbd} on $\T^d$ with initial condition $\rhove_0 = \rho_0$ for some $\rho_0 \in C^1(\T^d) \cap \mathcal{P}_C$. 
For $\varepsilon>0$ there exists time $T_\varepsilon>0$ such that $T_\varepsilon \to \infty$ as $\varepsilon \to 0$ and  dynamics \eqref{eqn:gfkerbd} with initial condition $\rho_0$ has a unique positive solution $\rhove \in  C^1([0, T_\varepsilon], L^1(\T^d))$. 
\end{theorem}
\begin{proof}
Let $\varepsilon>0$. By local well-posedess of Lemma \ref{lem:loc-existBDe} we know that a unique positive solution exists on some time interval $[0,T_0)$. 
Consider the time $T_\varepsilon$ defined in the proof of  Lemma \ref{lem:upperbdgradw}. We claim that the solution of \eqref{eqn:gfkerbd} exists until at least on $[0,T_\varepsilon)$.
 Namely if the maximal time of existence, $T$ is less than $T_\varepsilon$, then by
by Lemma \ref{lem:upperbdrhovet} $\rhove$ is uniformly bounded from below and above by positive constants. Thus there exists $\tilde C>1$ such that $\rhove \in \mathcal P_{\tilde C}$ for all $t \in [0,T)$. By applying the local existence \ref{lem:loc-existBDe} starting at time $\tau$ close enough to $T$ we can extend the solution beyond $T$ and obtain contradiction. 
\end{proof}

In order to study the convergence of \eqref{eqn:gfkerbd} to \eqref{eqn:purebd} as $\varepsilon\to 0$, we will rely on their gradient flow structure, which we investigate next.

\begin{definition} \label{def:convex}
Let $\mathcal{G}$ be a functional defined on $\mathcal{P}_C$. We say $\mathcal{G}(\rho)$ is $\lambda$-geodesically convex in $\mathcal{P}_C$ with respect to $d_{SH}$ if, for any $\rho_0,\rho_1\in \mathcal{P}_C$, let $(\rho_t)_{t\in[0,1]}$ be the $d_{SH}$-geodesics connecting $\rho_0$ to $\rho_1$, then \begin{equation}\label{eqn:dSHsemiconv}
    \mathcal{G}(\rho_1) -\mathcal{G}(\rho_0)- \frac{\ud}{\ud t} \mathcal{G}(\rho_t)\Big|_{t=0} \ge \frac{\lambda}{2}d^2_{SH}(\rho_0,\rho_1).
\end{equation} 
\end{definition}
Note that the above definition does not require $\mathcal{P}_C$ itself to be a geodesically convex set of $d_{SH}$. As long as $\rho$ is bounded above and below away from 0, both the relative entropy $\calF$ and regularized entropy $\fve$ are geodesically semiconvex with respect to $d_{SH}$, which is why we introduced the restricted submanifold $\mathcal{P}_C$. For general results regarding displacement convexity with respect to Hellinger-Kantorovich distance, we refer the readers to  \cite{liero2022fine}.
\begin{lemma}\label{lem:geosemiconvex}
Let $\varepsilon>0$ and $\Kve$ satisfy Assumption \ref{assump:kernel}. For any $C>1$, if $\pi \in \mathcal{P}_C$, then both $\calF(\rho) = \int_{\mathbb{T}^d} \rho\log \frac{\rho}{\pi}$ and $\fve(\rho)$ are $\lambda$-geodesically convex with respect to $d_{SH}$ in $\mathcal{P}_C$, for some $\lambda = \lambda(C,\varepsilon)\in \R$.
\end{lemma}
\begin{proof}
Consider $\rho_0,\rho_1 \in \mathcal{P}_C$ and $\rho_0\neq \rho_1$. Let us recall the $d_{SH}$-geodesics from $\rho_0$ to $\rho_1$  given in the expression in \cite[Lemma 2.7]{laschos2019geometric}:
    \begin{equation}\label{eqn:dSHgeodesic}
        \rho_t = \frac{\tilde\rho_{\beta_t}}{r_{\beta_t}}, \, \textrm{ with }\beta_t = \frac{\sin\big(td_{SH}(\rho_0,\rho_1)/2\big)}{\sin\big(td_{SH}(\rho_0,\rho_1)/2\big)+\sin\big((1-t)d_{SH}(\rho_0,\rho_1)/2\big)} \textrm{ and } r_t = \int_{\T^d} \tilde\rho_t,
    \end{equation}
and $\tilde\rho_t$ is the $d_H$-geodesics given in explicit form \eqref{eqn:geodesicdH}. Since $d_H(\rho_0,\rho_1)\le 2\sqrt{2}$, we can obtain from \eqref{eqn:conicdSH} that $d_{SH}(\rho_0,\rho_1) \le \pi$,  hence $\beta_t$ is well-defined and increases from 0 to 1. We may also obtain
\begin{equation}\label{eqn:bdrt}
    r_t = \int_{\T^d} \big(t\sqrt{\rho_1}+(1-t)\sqrt{\rho_0}\big)^2=1-\frac{t(1-t)}{4}d_H^2(\rho_0,\rho_1) \in \left[\frac{1}{2},1 \right],
\end{equation}which indicates $ \rho_t\in \mathcal{P}_{2C}$ 
for all $t\in[0,1]$. Thus, one can explicitly calculate that for $\calF(\rho) = \KL(\rho|\pi)$,
\begin{align*}\frac{\ud^2}{\ud t^2} \calF(\rho_t) & = \frac{\ud^2}{\ud t^2}\int_{\T^d} \rho_t \log \frac{\rho_t}{\pi} = \frac{\ud}{\ud t}\int_{\T^d}\partial_t\rho_t (\log\frac{\rho_t}{\pi}+1) = \int_{\T^d} \Big(\log \frac{\rho_t}{\pi}\frac{\ud^2}{\ud t^2}\rho_t + \frac{1}{\rho_t}(\frac{\ud}{\ud t}\rho_t)^2\Big).\end{align*}
The contribution from the second term is already non-negative; therefore in view of the pointwise bounds of $\rho_t$ and $\pi$, we only need to prove 
\begin{equation}\label{eqn:dSHcurvebd}
\int_{\T^d}\Big|\frac{\ud^2}{\ud t^2}\rho_t \Big|\lesssim d_{SH}^2(\rho_0,\rho_1).
\end{equation}
Taking second derivative on \eqref{eqn:dSHgeodesic}, using chain rule and quotient rule, we obtain
\begin{equation*}
    \frac{\ud^2}{\ud t^2}\rho_t = \Big(\frac{\tilde \rho_s''}{r_s} - \frac{2\tilde\rho_s' r_s'}{r_s^2} - \frac{\tilde\rho_sr''_s}{r_s^2}+\frac{2\tilde\rho_s (r_s')^2}{r_s^3} \Big)\Big|_{s=\beta_t} (\beta'_t)^2 + \Big(\frac{\tilde \rho_s'}{r_s}  - \frac{\tilde\rho_sr'_s}{r_s^2}\Big)\Big|_{s=\beta_t} \beta''_t.
\end{equation*}
We complete the proof of \eqref{eqn:dSHcurvebd} by combining the following facts:
\begin{align*}
     \int \! \tilde\rho_s &= r_s\in \left[\frac{1}{2},1 \right], \\
     \int \!|\tilde\rho_s''| & = |r_s''|= 2 \int (\sqrt{\rho_1}-\sqrt{\rho_0})^2 \le \frac{1}{2}d_{SH}^2(\rho_0,\rho_1), \\ 
      \int \!|\tilde \rho'_s| & = 2 \int \sqrt{\tilde\rho_s}|\sqrt{\rho_1}-\sqrt{\rho_0}| \le 2\Big(\int \tilde\rho_s\Big)^\frac{1}{2}\Big(\int (\sqrt{\rho_1}-\sqrt{\rho_0})^2\Big)^\frac{1}{2} \le  \sqrt{r_s} d_{SH}(\rho_0,\rho_1), \\
       |r'_s| & = \left|\frac{2s-1}{4}\right| d_H^2(\rho_0,\rho_1) \le \frac{1}{4}d_{SH}^2 (\rho_0,\rho_1), \\
        \beta_t' & = \frac{d_{SH}(\rho_0,\rho_1)\sin\big(d_{SH}(\rho_0,\rho_1)/2\big)}{2\Big(\sin\big(td_{SH}(\rho_0,\rho_1)/2\big)+\sin\big((1-t)d_{SH}(\rho_0,\rho_1)/2\big)\Big)^2}\, \in [0,\frac{\pi}{2}],  \\ 
       |\beta_t''| &  =\frac{d_{SH}^2(\rho_0,\rho_1)\sin\big(d_{SH}(\rho_0,\rho_1)/2\big)\Big|\cos\big(td_{SH}(\rho_0,\rho_1)/2\big)-\cos\big((1-t)d_{SH}(\rho_0,\rho_1)/2\big)\Big|}{2\Big(\sin\big(td_{SH}(\rho_0,\rho_1)/2\big)+\sin\big((1-t)d_{SH}(\rho_0,\rho_1)/2\big)\Big)^3} \!\!\\
         &  \le Cd_{SH}^2(\rho_0,\rho_1).   
\stepcounter{equation} \tag{\theequation}\label{eqn:dSHdscbds}
\end{align*}
\smallskip

We now turn our attention to the regularized energy $\fve$. Taking second derivative, following the same calculation as above,  one has 
\begin{align*}
    \frac{\ud^2}{\ud t^2} \int_{\T^d} \fve(\rho_t) & = \int_{\T^d} \log \frac{\Kve*\rho_t}{\pi}\frac{\ud^2}{\ud t^2}\rho_t + \int_{\T^d} \frac{\rho_t}{\Kve*\rho_t} \Kve* \frac{\ud^2 }{\ud t^2}\rho_t + 2\int_{\T^d} \frac{\frac{\ud \rho_t}{\ud t}}{\Kve*\rho_t} \Kve* \frac{\ud \rho_t}{\ud t} \\ 
    & \qquad - \int_{\T^d} \frac{\rho_t}{(\Kve*\rho_t)^2} \left(\Kve*\frac{\ud\rho_t}{\ud t} \right)^2 \stepcounter{equation} \tag{\theequation} \label{eqn:curvfve}  \\ 
    & \ge -(4\log 2C+C^2) \int_{\T^d} \Big|\frac{\ud^2}{\ud^2 t}\rho_t \Big|-2C \int_{\T^d} \frac{\ud \rho_t}{\ud t} \Kve* \frac{\ud \rho_t}{\ud t}- C^3\int_{\T^d} \left(\Kve*\frac{\ud\rho_t}{\ud t} \right)^2 \\ 
    & \ge  -(4\log 2C+C^2) \int_{\T^d} \Big|\frac{\ud^2}{\ud^2 t}\rho_t\Big| -(2C+C^3) \int_{\T^d} \left(\frac{\ud \rho_t}{\ud t} \right)^2.
\end{align*}
Here the second term on the right side of the second line can be bounded using Assumption \ref{assump:kernel} on $\Kve$ and Jensen's inequality: 
\begin{align*}
    \int_{\T^d} \frac{\ud \rho_t}{\ud t} \Kve* \frac{\ud \rho_t}{\ud t} = \int_{\T^d}  \left(\xieps* \frac{\ud \rho_t}{\ud t}\right)^2\le \int_{\T^d}  \xieps* \left(\frac{\ud \rho_t}{\ud t}\right)^2 = \int_{\T^d} \left(\frac{\ud \rho_t}{\ud t}\right)^2.
\end{align*}
The third term  on the right side of the second line can be treated in the identical way. In view of \eqref{eqn:dSHcurvebd}, it remains to show
\begin{equation*}
    \int_{\T^d} \left(\frac{\ud \rho_t}{\ud t} \right)^2 \lesssim d_{SH}^2(\rho_0,\rho_1).
\end{equation*}
which is true since $\rho_0,\rho_1\in \mathcal{P}_C$ and we may apply the bounds in \eqref{eqn:dSHdscbds} to
\begin{equation*}
    \frac{\ud}{\ud t}\rho_t =\Big(\frac{\tilde \rho_s'}{r_s}  - \frac{\tilde\rho_sr'_s}{r_s^2}\Big)\Big|_{s=\beta_t} \beta'_t.
\end{equation*}
\end{proof}

We now introduce the elements needed to consider \eqref{eqn:gfkerbd} as  $d_{SH}$ gradient flows. The space of all probability densities on $\T^d$ equipped with $d_{SH}$ is a complete Riemannian manifold with corners. The explicit formulas for geodesics \eqref{eqn:dSHgeodesic} ensure that for all $\rho_0, \rho_1 \in \mathcal{P}_C$ and all $t \in [0,1]$ the geodesics $\rho_t \in \mathcal{P}_{2C}$. 
The tangent spaces are already identified in the earlier work \cite{gallouet2017jko}:
    For $\rho \in \mathcal{P}_C$, the tangent space at $\rho$ with respect to $d_{SH}$ can be identified with 
\begin{equation} \label{eq:tanSH}
T_{SH, \rho} = L^2_0(\rho):=\left\{\phi\in L^2(\rho) \::\: \int_{\T^d}\phi \ud \rho=0\right\}.
\end{equation}

We now compute the geometric logarithm (inverse of the exponential map) on the space of probability measures endowed with $d_{SH}$ geometry and a point $\rho\in\mathcal{P}_C$ for some $C$. In other words we compute the tangent vector to the unit-time geodesic connecting two measures. 
\begin{lemma} 
For $\rho,\mu\in \mathcal{P}_C$, we have
    \begin{equation} \label{eqn:logdSH}
        \ln_\rho^{d_{SH}}\mu= \frac{d_{SH}(\rho,\mu)/2}{\sin\big(d_{SH}(\rho,\mu)/2\big)} 
          \left[2 \left(\sqrt{\frac{\mu}{\rho}}-1\right) + \frac{d_{SH}(\rho,\mu)^2}{4} \right]  .
    \end{equation}
\end{lemma}
\begin{proof}
    Let $\{\rho_t\}_{t\in[0,1]}$ be the $d_{SH}$-geodesics connecting from $\rho$ to $\mu$. We refer the readers to the expression in \eqref{eqn:dSHgeodesic},
    then apply the chain rule and we obtain
    \begin{align*}
         \ln_\rho^{d_{SH}}\mu & = \frac{\frac{\ud}{\ud t}\rho_t}{\rho_t}\Big|_{t=0} = \frac{\tilde\rho_0' r_0-\tilde\rho_0 r'_0}{r_0^2\tilde\rho_0}\beta'_0 \\ & = \frac{d_{SH}(\rho,\mu)/2}{\sin\big(d_{SH}(\rho,\mu)/2\big))}\left(2\left(\sqrt{\frac{\mu}{\rho}}-1\right) + \frac{d_{SH}^2(\rho,\mu)}{4}\right). \stepcounter{equation} \tag{\theequation} \label{eqn:lncalc}
    \end{align*}
We note that due to the remark after \eqref{eqn:conicdSH}, $\sin\big(d_{SH}(\rho,\mu)/2\big) >0$.
\end{proof}

Note that the tangent spaces are Hilbert spaces, which allows us to define the subdifferentials as the classical Fr\'echet subdifferentials:
\begin{definition}\label{def:subdiff}(Subdifferential)
    Let $\mathcal{G}: \mathcal P(\T^d) \to \R$ be proper and lower semicontinuous with respect to $d_{SH}$. Assume $\rho>0$ on $\T^d$,  $\mathcal{G}(\rho)<\infty$, and $\zeta \in L^2_0(\rho)$. We say that $\zeta$ is in the subdifferential of $\mathcal{G}$ at $\rho$, and write $\zeta \in \partial_{SH} \mathcal{G}(\rho)$, if for any $\phi \in L^2_0(\rho)$ such that $(1+\phi)\rho \geq 0$,
    \begin{equation}\label{eqn:dSHsubdiff}
        \mathcal{G}((1+\phi)\rho)- \mathcal{G}(\rho) \ge \int_{\T^d} \zeta  \phi  \, \rho \, dx
        + o(\|\phi \|_{L^2(\rho)}).
    \end{equation}
We note that $\phi$ belongs is the  Fr\'echet differential if the inequality above is replaced by equality. 
\end{definition} 
We  note that $o(\|\phi \|_{L^2(\rho)}) = o(d_{SH}(\rho, (1 + \phi) \rho))$.


\begin{lemma}\label{lem:struppgr}
    Consider $\rho \in \mathcal{P}_C$ for some $C >1$. Suppose $\mathcal{G} : \mathcal{P} \to \R$ is proper and lower semicontinuous.
    Assume that there exists $\zeta \in L^\infty(\T^d)$ such that for all $h \in L^2(\rho)$ which is essentially bounded from below (i.e. $h_- \in L^\infty$), the first variation of $\mathcal{G}$ in direction $h$ exists and is of the form $\frac{\delta \mathcal{G}}{\delta \rho}\Big|_\rho [h] = \int \zeta h dx $. Furthermore assume  that $\mathcal{G}$ is $d_{SH}$ $\lambda$-geodesically convex in $\mathcal{P}_C$ for some $\lambda\in \R$. Then 
    \begin{equation}\label{eqn:subdiffdelG}
          \partial_{SH} \mathcal{G}(\rho) = \left\{\zeta  - \int_{\T^d} \zeta \rho\right\}.
    \end{equation}
\end{lemma}
\begin{proof}
Let us first show that any element of  $ \partial_{SH} \mathcal{G}(\rho)$ must be equal to $\zeta  - \int_{\T^d} \zeta \rho$.  Assume that $\xi \in \partial_{SH} \mathcal{G}(\rho) $, namely it satisfies 
\eqref{eqn:dSHsubdiff}. For $\phi \in  L^2_0(\rho) \cap L^\infty(\T^d)$ and $t >0$ consider $t \phi$ playing the role of $\phi$ in  \eqref{eqn:dSHsubdiff}, dividing by $t$ and taking the limit as $t \to 0$ provides
\[ \int \zeta \phi \rho  \geq \int \xi \phi \rho.\]
Doing the same for $t<0$ yields 
\[ \int \zeta \phi \rho  \leq \int \xi \phi \rho.\]
Hence $\int \zeta \phi \rho = \int \xi \phi \rho$ for all $\phi \in  L^2_0(\rho) \cap L^\infty(\T^d)$. Thus $\xi =\zeta  - \int_{\T^d} \zeta \rho$ since $\xi \in L^2_0(\rho)$.

\smallskip

Let us now show that $\zeta  - \int_{\T^d} \zeta \rho \in \partial_{SH} \mathcal{G}(\rho)$. Let $\{\rho_t\}_{t \in [0,1]}$ be the $d_{SH}$ geodesics connecting $\rho$ and $(1 + \phi) \rho$. Note that, by 
\eqref{eqn:logdSH}, $\frac{d}{dt} \rho_t(0) $ is essentially bounded from below.
Using that $\zeta \in L^\infty$ we obtain
\begin{align*}
\frac{\ud}{\ud t}\Big|_{t=0} \mathcal{G}(\rho_t) & = \int \zeta  \ln_{\rho}^{SH} ((1 + \phi) \rho) \, \zeta \rho dx \\
& = \frac{d_{SH}(\rho,(1 + \phi) \rho)/2}{\sin\big(d_{SH}(\rho,(1 + \phi) \rho)/2\big))}\, 2   \int \left(\sqrt{\frac{(1 + \phi) \rho}{\rho}}-1\right) \, \zeta \rho dx + o(d_{SH}(\rho, (1 + \phi) \rho))\\
& = \int \zeta \phi dx + o(d_{SH}(\rho, (1 + \phi) \rho)) = \int \left( \zeta  - \int_{\T^d} \zeta \rho \right)\phi  dx + o(\|\phi\|_{L^2(\rho)}).
\end{align*}
Combining this with the $\lambda$ convexity of $\mathcal{G}$ implies that $\zeta  - \int_{\T^d} \zeta \rho \in \partial_{SH} \mathcal{G}(\rho)$.
\end{proof}

The above lemma allows us to identify the subgradients of $\calF$ and $\fve$. Since the subgradients are singletons we identify each set with its only element. That is, for $\rho \in \mathcal P_C$,
\begin{align*}
\partial_{SH} \calF(\rho)& =  \log \frac{\rho}{\pi} - \int_{T^d} \log \frac{\rho}{\pi} \, \rho \\
\partial_{SH} \mathcal F_\varepsilon(\rho)& =    \log \left(\frac{\Kve*\rho}{\pi}\right) + \Kve * \left( \frac{ \rho}{\Kve* \rho} \right) - \! \int \log \left(\frac{\Kve*\rho}{\pi}\right) \rho -1
\end{align*}

The last step to rigorously identify \eqref{eqn:purebd} and \eqref{eqn:gfkerbd} as $d_{SH}$-gradient flows is to show that they are \emph{curves of maximal slope.} Let us recall that for a curve $\rho_t$, the metric derivative is given by 
\[|\partial_t \rho_t| = \lim_{s\to t} \frac{d_{SH}(\rho_s,\rho_t)}{|s-t|},\] and the metric slope is defined as \cite[(10.0.9)]{ambrosio2008gradient} \begin{equation} \label{eqn:defmetricslope}|\partial_{SH} \mathcal{G} (\rho)| = \limsup_{d_{SH}(\nu,\rho) \to 0} \frac{\left(\mathcal{G}(\rho)-\mathcal{G}(\nu)\right)_+}{d_{SH}(\nu,\rho)}.\end{equation}
We say that a path $\rho \in AC([0,T],  (\mathcal P_C, d_{SH}))$ is a \emph{gradient flow} solution of \eqref{eqn:genFRgf} if $\partial_t \rho_t = -\rho_t\partial_{SH} \mathcal G(\rho_t)$ for a.e. $t  \in [0,T]$. 
We say that $\rho_t \in AC([0,T],  (\mathcal P_C, d_{SH}))$ is a \emph{curve of maximal slope} for functional $\mathcal G$ if $\frac{d}{dt} \mathcal G(\rho) \leq - \frac12 \left| \partial_t \rho_t\right|^2 - \frac12 \left| \partial_{SH} \mathcal G(\rho) \right|^2$
where $\left| \partial_t \rho_t\right|$  is the $d_{SH}$ metric derivative and 
$\left| \partial_{SH} \mathcal G(\rho) \right|$ is the metric slope. Note that in Definition \ref{def:subdiff} we already require $\partial_{SH} \mathcal{G}(\rho_t) \in L^\infty(\mathbb{T}^d)$, and $\rho_t\in \mathcal{P}_C \subset L^\infty(\mathbb{T}^d)$, hence $\partial_t \rho_t$ is well-defined in the classical sense.

\begin{lemma} \label{lem:metric slope}
Consider $\rho \in \mathcal{P}_C$ for some $C>1$ and suppose that the assumptions of Lemma \ref{lem:struppgr} are satisfied. Assume furthermore that 
$\partial_{SH} \mathcal{G}(\mu)$ is continuous in $L^\infty(\T^d)$ in an $L^\infty$ neighborhood of $\rho$. Then the metric slope satisfies 
\[ |\partial_{SH} \mathcal{G}(\rho) | = 
\|\partial_{SH} \mathcal{G}(\rho)\|_{L^2(\rho)}. \]
\end{lemma}
\begin{proof}
Let $\zeta = \partial_{SH} \mathcal{G}(\rho)$. Let $\mu_t = (1-t \zeta) \rho$. The fact that $|\partial_{SH} \mathcal{G}(\rho) | \leq
\|\partial_{SH} \mathcal{G}(\rho)\|_{L^2(\rho)}$ follows by considering the limit along $\mu_t \to \rho$.  To show the opposite inequality note that 
\[ \mathcal{G}(\rho) -  \mathcal{G}(\mu_t) \geq  \int \zeta_t (\rho - \mu_t) + o(d_{SH}(\rho, \mu_t))  = t \int \zeta_t \zeta + o(d_{SH}(\rho, \mu_t)) \]
where $\zeta_t = \partial_{SH} \mathcal{G}(\mu_t)$. 
Noting that $d_{SH}(\rho, \mu_t) = t \| \zeta \|_{L^2(\rho)} + o(d_{SH}(\rho, \mu_t)) $ and using the continuity of $\zeta_t$ implies that
\[ \limsup_{t \to 0+} \frac{\mathcal{G}(\rho) -  \mathcal{G}(\mu_t)}{d_{SH}(\rho, \mu_t)}  \geq \|\zeta\|_{L^2(\rho)}.   \]
\end{proof}

With the above preparation, we are able to rigorously identify \eqref{eqn:gfkerbd} as the $d_{SH}$ gradient flow of $\fve$. The existence results of Lemma \ref{lem:wellposeeps0} and Theorem \ref{thm:wlBDeps} imply that during the interval of existence  $\partial_t \rhove = -\rhove\partial_{SH} \fve(\rhove)$. Furthermore the solutions are in $AC([0,T], (\mathcal P_C, d_{SH}))$, which can be verified by direct check based of solution formula of Lemma \ref{lem:wellposeeps0} for $\rho_t$, and by $\rhove \in C^1([0,T], L^2(\T^d))$ using Theorem \ref{thm:wlBDeps}. Thus $\rho$ and 
 $\rhove$ are gradient flows of the respective equations, and consequently curves of maximal slope.
 \smallskip

Due to the semiconvexity of the functionals, the solutions also satisfy an evolution variational inequality (Chapter 11 of \cite{ambrosio2008gradient} for Wasserstein gradient flows, and \cite{muratori20} for general metric spaces). 
This implies that gradient flow solutions are unique. In particular while our notion of $\lambda$ convexity is restricted, the proof of quantitative stability of Theorem 11.1.4 of \cite{ambrosio2008gradient} carries over.  
Thus gradient flow solutions 
 coincide with the solutions of Lemma \ref{lem:wellposeeps0} and  Theorem \ref{thm:wlBDeps}, respectively. 


\subsection{$\Gamma$-convergence of gradient flows} \label{sec:gamma}

The next question is whether the dynamics \eqref{eqn:gfkerbd} converges in any sense to the idealized dynamics  \eqref{eqn:purebd} as $\varepsilon \to 0$. The natural notion to study is the $\Gamma$-convergence of gradient flows \`a la Sandier-Serfaty \cites{sandier2004gamma, serfaty2011gamma}. We will show a proof for $\T^d$ with a compactly supported kernel. The following theorem, essentially a rephrase of \cite[Theorem 2]{serfaty2011gamma} but adapted to our setting, summarizes the conditions we need to verify for the $\Gamma$-convergence of gradient flow.

\begin{theorem}\label{thm:gammaconvgen} \cite{serfaty2011gamma}*{Theorem 2}
Let $\rhove_t$ be solutions of \eqref{eqn:gfkerbd} that are curves of maximal slopes.
Suppose the initial conditions are well-prepared, in the sense that as $\varepsilon\to 0$, \begin{equation*}
    \rhove_0 \xrightarrow{d_{SH}} \rho_0, \, \textrm{ and }\, \fve(\rhove_0) \rightarrow \calF(\rho_0).
\end{equation*}Suppose also that as $\varepsilon\to 0$, we have $\rhove_t \xrightarrow{d_{SH}}\nu_t$ 
for almost every $t$, as well as the following conditions:   
\begin{listi}
 \item $\displaystyle{\liminf_{\varepsilon\to 0}}\int_0^t |\partial_t \rhove_s|^2 ds \ge \int_0^t|\partial_t\nu_s|^2 ds$,
    \item $\displaystyle{\liminf_{\varepsilon\to 0}}\, \fve(\rhove_t) \ge \calF(\nu_t)$,
    \item The slopes $\partial_{SH} \fve(\rhove_t)$ and $\partial_{SH} \calF(\nu_t)$ are strong upper gradients, and \[\displaystyle{\liminf_{\varepsilon\to 0}}\int_0^t |\partial_{SH} \fve(\rhove_s)|^2 ds \ge \int_0^t|\partial_{SH} \calF(\nu_s)|^2 ds,\]
\end{listi}
then $\nu_t$ must be the solution of gradient flow \eqref{eqn:purebd} with energy functional $\calF$ and initial condition $\nu_0=\rho_0$.
\end{theorem}
\begin{theorem}\label{thm:bdgammaconv}
 Under the same assumptions as Theorem \ref{thm:wlBDeps}, for any fixed $T$, $(\rhove_t)_{0\le t \le T}$ converges in $d_{SH}$ to $(\rho_t)_{0\le t \le T}$ the solution of \eqref{eqn:purebd} with the same initial condition $\rhove_0=\rho_0$. In particular, $\lim_{\varepsilon \to 0} \fve(\rhove_t) = \calF(\rho_t)$.
\end{theorem}
\begin{proof}
The plan of our proof is to first use Arzela-Ascoli Theorem to identify the limiting sequence $\nu_t$, then verify the three conditions in Theorem \ref{thm:gammaconvgen}.

\smallskip

Notice $\nabla \rhove_t = \rhove_t \nabla \wvet$, hence the combination of Lemmas \ref{lem:upperbdgradw} and \ref{lem:upperbdrhovet}, as well as $\pi \in \mathcal{P}_C$, yield that $\nabla \rhove_t$ is uniformly bounded for any $t\in[0,T]$, when $\varepsilon$ is sufficiently small. Thus, since $\rhove_t$ is also uniformly bounded (c.f. Lemma \ref{lem:upperbdrhovet}), we may invoke Arzela-Ascoli Theorem, so that there exists a subsequence, still denoted as $\rhove_t$, that converges uniformly to some function $\nu_t$ as $\varepsilon\to 0$.

\smallskip

We now verify the three conditions of Theorem \ref{thm:gammaconvgen}. The proof of (ii) is straightforward, since by $\Gamma$-convergence of the energy functional \cite{carrillo2019blob}*{Theorem 4.1} (note that convergence in $d_{SH}$ implies convergence in $W_2$ on $\T^d$), $\displaystyle{\liminf_{\varepsilon\to 0}}\, \fve(\rhove_t) \ge \calF(\nu_t)$, and trivially $\rhove_t \in \mathcal{P}_2(\T^d)$ on bounded domain.

\smallskip

We now prove (i). The proof is standard and follows from the arguments in \cite{craig2016convergence}*{Theorem 5.6}. Assume without loss of generality that there exists $0 \le C < \infty$ so that
\[C = \liminf_{\varepsilon \to 0} \int_0^T |\partial_t \rhove_t |^2 dt.\]
Choose a subsequence $|\partial_t\tilde{\rho}_t^{(\varepsilon)}|$ so that $\lim_{\varepsilon \to 0} \int_0^T |\partial_t \tilde{\rho}_t^{(\varepsilon)}  |^2 dt = C$. Then $|\partial_t \tilde{\rho}_t^{(\varepsilon)}|$ is bounded in $L^2(0,T)$ so, up to a further subsequence, it is weakly convergent to some $v(t)\in L^2(0,T)$. Consequently, for any $0\le s_0\le s_1\le T$,
\[\lim_{\varepsilon\to 0}\int_{s_0}^{s_1} |\partial_t \tilde{\rho}_t^{(\varepsilon)} | dt = \int_{s_0}^{s_1} v(t) dt. \]
By taking limits in the definition of the metric derivative and using the lower semi-continuity of $d_{SH}$ with respect to weak-* convergence (see the proof in \cite[Theorem 5]{kondratyev2016new} for $d_H$, which also applies to $d_{SH}$ using conic structure \eqref{eqn:conicdSH}), 
\[d_{SH}(\tilde{\rho}_{s_0}^{(\varepsilon)},\tilde{\rho}_{s_1}^{(\varepsilon)}) \le \int_{s_0}^{s_1} |\partial_t \tilde{\rho}_t^{(\varepsilon)} | dt \Rightarrow d_{SH}(\nu_{s_0},\nu_{s_1}) \le \int_{s_0}^{s_1} v(t) dt. \]
  By \cite{ambrosio2008gradient}*{Theorem 1.1.2}, this implies that $|\partial_t \nu_t| \le v(t)$ for a.e. $t\in (0, T)$. Thus, by the lower semicontinuity of the $L^2(0,T)$ norm with respect to weak convergence,
\[\liminf_{\varepsilon\to 0} \int_0^T |\partial_t \rhove_t |^2 dt = \lim_{\varepsilon \to 0} \int_0^T |\partial_t \tilde{\rho}_t^{(\varepsilon)} |^2 dt \ge \int_0^T v(t)^2 dt \ge \int_0^T |\partial_t \nu_t|^2 dt.\]

\smallskip

Regarding (iii), we first claim here that $\rhove_t$ converges uniformly to $\nu_t$ implies that $\Kve * \rhove_t$ also converges uniformly to $\nu_t$ as $\varepsilon\to 0$. Indeed, 
\begin{align*}
    |\Kve*\rhove_t(x) - \nu_t(x) | & =  \Big|\int_{\mathbb{T}^d} \Kve(x-y)(\rhove_t(y) -\nu_t(x)) \ud y \Big| \\ & \le \Big|\int_{\mathbb{T}^d} \Kve(x-y)(\rhove_t(y) -\rhove_t(x)) \ud y  \Big| \\ & \qquad +\Big|\int_{\mathbb{T}^d} \Kve(x-y)(\rhove_t(x) -\nu_t(x)) \ud y \Big| \\ & \le \sup_{z\in \mathbb{T}^d} |\nabla \rhove_t(z)| \int_{\mathbb{T}^d} \Kve(x-y)|x-y|\ud y + \sup_{z\in \mathbb{T}^d} |\rhove_t(z)-\nu_t(z)|.
\end{align*} The first term goes to zero since $|\nabla \rhove_t(z)|$ is uniformly bounded in $z$ and $\varepsilon$, while the integral is $\varepsilon M_1(K)\to 0$, while the second term also goes to zero due to uniform convergence of $\rhove_t$ to $\nu_t$.
Moreover, by \eqref{eqn:infbdrve}, $\nu_t(x)$ is bounded and away from zero for all $x\in \mathbb{T}^d$ and $t\in[0,T]$, which implies that $\frac{\rhove_t}{\Kve* \rhove_t}$ and consequently $\Kve*\frac{\rhove_t}{\Kve* \rhove_t}$ converge uniformly to 1 as $\varepsilon \rightarrow 0$. Consequently,  the inequality 
\begin{multline}\label{eqn:goalmetricslope}
\liminf_{\varepsilon\to 0} \int_0^T\! \int \rhove_t\left( \log \frac{\Kve*\rhove_t}{\pi} + \Kve * \left( \frac{ \rhove_t}{\Kve* \rhove_t}\right) - \int \rhove_t \log \frac{\Kve*\rhove_t}{\pi} -1\right)^2 \\ 
\ge \int_0^T \!\int \nu_t \left(\log \frac{\nu_t}{\pi} - \int \nu_t \log \frac{\nu_t}{\pi} \right)^2\end{multline}
holds by Fatou's lemma. Now, by \cite[Corollary 2.4.10]{ambrosio2008gradient}, since 
$\calF$ and $ \fve$ are geodesically semiconvex on $\mathcal P_C$,
$|\partial_{SH} \mathcal F(\rho)|$ and
$|\partial_{SH} \fve(\rhove)|$ are strong upper gradients.
By Lemma \ref{lem:struppgr} and Lemma \ref{lem:metric slope} we have that for $\mathcal{G} = \calF$ or $\fve$,
\begin{equation}\label{eqn:dSHslope}
    |\partial_{SH} \mathcal{G} (\rho)| = \left(\int_{\T^d} \rho\left(\frac{\delta\mathcal{G}}{\delta\rho} - \int_{\T^d} \rho\frac{\delta\mathcal{G}}{\delta\rho} \right)^2\right)^\frac{1}{2}.
\end{equation} Hence, in view of \eqref{eqn:goalmetricslope}, we can verify condition (iii) of Theorem \ref{thm:gammaconvgen}. This allows us now to conclude:  since all three conditions of Theorem \ref{thm:gammaconvgen} are now fulfilled, by $\Gamma$-convergence, $\nu_t$ must be a solution of \eqref{eqn:purebd} with initial condition $\rho_0$, and therefore by the uniqueness result established in Lemma \ref{lem:wellposeeps0}, must be $\rho_t$.
\end{proof}

\begin{remark}
We would like to comment here that the above strategy does not apply to the whole space $\R^d$, since we could not assume that $\nabla V \in L^\infty(\R^d)$ or $\wvet$ being globally Lipschitz continuous on $\R^d$. Moreover, the ratio $\frac{\rho}{\Kve*\rho}$ may be very close to 0 at infinity, unlike Lemma \ref{lem:rhokverhoclose} which says the ratio is always $1\pm O(\varepsilon)$, making it difficult to compare \eqref{eqn:gfkerbd} with \eqref{eqn:purebd}. Finally $\fve$ might not be displacement semiconvex when $\rho$ is close to zero, which is unavoidable in the whole space.
\end{remark}

\begin{remark}
If the dynamics \eqref{eqn:gfkerbd} has an initial condition $\rho_0 = \sum_{i=1}^N m_i(0) \delta_{x_i}$ for some $x_i \in \R^d,i=1,\ldots,N, \, m_i(0)>0, \, \sum_{i=1}^N m_i(0)=1$, then \eqref{eqn:gfkerbd} has a solution of the form $\rho_t =  \sum_{i=1}^N m_i(t) \delta_{x_i}$ where the masses $m_i(t)$ satisfy the following ODE: 
\begin{align} \label{eqn:bdkergfn}
\begin{split}
    \frac{d m^i}{dt} =  & 
    - m_i \left[ \log\left( \sum_{j=1}^N m_j \Kve(x^i-x^j)\right) - \log \pi(x^i)  +
    \sum_{j=1}^N \frac{m_j \Kve(x^i-x^j)}{\sum_{k=1}^N m_\ell \Kve (x^j-x^k)} \right. \\
    & \left. \phantom{ - m_i [}\,- \sum_{\ell=1}^N m_\ell \log \left( \sum_{j=1}^N m_j \Kve(x^{\ell}-x^j)\right) + \sum_{\ell=1}^N m_\ell \log \pi(x^\ell) -1 \right].
\end{split} 
\end{align}
The above ODE is obviously well-posed, since it is a finite-dimensional ODE with a trapping region, which is the probability simplex. On the other hand,  although we are unable to prove it, we believe the long-time well-posedness of \eqref{eqn:gfkerbd} with a smooth initial condition is true with more careful estimates. More specifically, we believe the solution can be approximated by certain minimizing movement scheme, at least on a compact domain as shown \cite{laschos2022evolutionary} in a different setting. We leave careful investigations along this line for future research. 
\end{remark}

\subsection{Convergence of asymptotic sets} \label{sec:asymptotics}
While $\fve$ defined in \eqref{eqn:funckerkl} may not have a unique minimizer, and the dynamics \eqref{eqn:gfkerbd} may not converge to a unique probability distribution as $t \to \infty$, the $\Gamma$-convergence of regularized gradient flows \eqref{eqn:gfkerbd} to \eqref{eqn:purebd} and the long-time convergence of the limiting gradient flow \eqref{eqn:purebd} guarantees that the long-time limiting set of \eqref{eqn:gfkerbd} is close to that of \eqref{eqn:purebd}, which is $\pi$. The goal of this section is to discuss some sufficient conditions where convergence of asymptotic sets of gradient flows hold.

\smallskip

We first present below Proposition \ref{prop:asyset}, which we state for general gradient flows in metric spaces. In particular the lemma can be applied to gradient flows in Wasserstein metric and has interesting consequences for 2-layer neural network training which we will discuss in Theorem \ref{thm:jmmasyconv}. 

\begin{proposition}\label{prop:asyset}
Let $E_\varepsilon$ and $E$ be energy functionals, and let $(\rhove_t)_{0\le t < T_\varepsilon^*}$ and $(\rho_t)_{0\le t <\infty}$ be continuous curves in a metric space, with metric $d$ and maximal existence time $T_\varepsilon^*$ and $\infty$ respectively, such that $E_\varepsilon$ is nonincreasing along $\rhove_t$ and $E$ is nonincreasing along $\rho_t$. 
Assume the following conditions hold:
\begin{listi}
    \item $E$ has a unique minimizer $\pi$, and $\rho_t$ converges  to $\pi$ as $t\to\infty$ in the sense that $E(\rho_t) \rightarrow E(\pi)$.
    \item  As $\varepsilon \to 0$, $\liminf_{\varepsilon\to 0} E_\varepsilon(\mu_\varepsilon)\ge E(\mu)$ for all sequences $\mu_\varepsilon \xrightarrow{d}\mu$.
   \item As $\varepsilon\to 0$, we have $T_\varepsilon^* \to \infty$. Moreover, for every $t\ge 0$, we have that $\rhove_t \to \rho_t$ in $d$ and $E_\varepsilon(\rhove_t) \rightarrow E(\rho_t)$.
 \item The sub-level sets of $E_\varepsilon$ are uniformly precompact in the following sense: There exists $\varepsilon_0>0$ such that for any $M\in(0,\infty)$ the set $\bigcup \left\{ E_\varepsilon^{-1} (-\infty, M) \::\: 0< \varepsilon < \varepsilon_0 \right\}$ is precompact. 
\end{listi}
Then we have the following:
\begin{lista}
    \item For any $\varepsilon >0$ and any time sequence $(T_\varepsilon)_\varepsilon$ such that $T_\varepsilon< T_\varepsilon^*$ and $\lim_{\varepsilon\to 0} T_\varepsilon = \infty$, we have $\rhove_{T_\varepsilon} \xrightarrow{d} \pi $ as $\varepsilon\to 0$.
    \item If $T_\varepsilon^*=\infty$ for any sufficiently small $\varepsilon$, then let 
\begin{equation}\label{eqn:defolim}
\mathcal{A}_\varepsilon := \left\{\rhove_\infty \: \Big| \: \exists \, 0\le t_1<t_2<\ldots<t_n<\ldots  \text{ s.t. } \lim_{n\to \infty} t_n=\infty \text{ and }  \lim_{n\to \infty}\rhove_{t_n}= \rhove_\infty \text{ in } d \right\}
\end{equation} be the $\omega$-limit set of $\rhove_t$, we have \[\mathcal{A}_\varepsilon \to \{\pi\} \text{ as } \varepsilon \rightarrow 0\] with respect to Hausdorff distance  corresponding to $d$. 
\end{lista} 
\end{proposition}


\begin{proof}
We only prove (a) since the proof for (b) is identical. Fix $\delta>0$, then there exists a $T>0$ such that \begin{equation*}
    E(\rho_T) \le E(\pi)+\delta.
\end{equation*} By assumption (iii), there exists some $\varepsilon_0 \geq \varepsilon_1>0$ such that for all $\varepsilon<\varepsilon_1$, we have $T_\varepsilon^* >T_\varepsilon>T$, and 
\begin{equation*}
    E_\varepsilon(\rhove_T) \le E(\rho_T)+\delta.
\end{equation*}
To prove the claim we argue by contradiction. Assume that  $\limsup_{\varepsilon \to 0} d(\rhove_{T_\varepsilon},\pi) \geq \lambda >0$, then along a subsequence (not relabeled) $\varepsilon \to 0$, there  exists  $\rhove_{T_\varepsilon} $ such that $d(\rhove_{T_\varepsilon}, \pi)>\lambda/2$. 

\smallskip

Using that $E_\varepsilon$ is nonincreasing along $\rhove_t$, we have
\begin{equation*}
    E_\varepsilon(\rhove_{T_\varepsilon}) \le E_\varepsilon(\rhove_T) \le E(\pi)+2\delta.
\end{equation*}
Using the compactness assumption (iv) it follows that 
$\rhove_{T_\varepsilon} \to  \sigma$ in metric $d$ along a further subsequence as $\varepsilon \to 0$ for some $\sigma$.
From lower-semicontinuity assumption (ii) follows that
\[E(\sigma) \le \lim\inf_{\varepsilon\to 0} E_\varepsilon(\rhove_{T_\varepsilon}) \le E(\pi)+2\delta.\] 
Since $\delta$ is arbitrary, we can take $\delta\to 0$ to obtain 
\[E(\sigma) \le E(\pi),\] 
which in turn gives $\sigma=\pi$ since $\pi$ is the unique minimizer of $E$.
On the other hand, 
\begin{equation*}
    d(\sigma,\pi) = \lim_{\varepsilon\to 0} d(\rhove_{T_\varepsilon},\pi) \ge \frac{\lambda}{2}.
\end{equation*} 
Contradiction.
\end{proof}

\noindent An application of Proposition \ref{prop:asyset} is the following Theorem \ref{thm:jmmasyconv}. Here let us recall the setting of two-layer neural network training in \cite{javanmard2020analysis}: the goal is to learn a concave function $f$ using a neural network, which is achieved by minimizing the risk functional in domain $\Omega$ with noise level $\tau>0$:
\begin{equation}\label{eqn:jmmfunctionallater} F^\delta(\rho^\delta) = \int_{\Omega} \left(\frac{1}{2}(K_\delta*\rho^\delta-f)^2+\tau \rho^\delta \log \rho^\delta\right) \ud x.\end{equation} As $\delta\to 0$, $F^\delta$ is close to the limiting functional \begin{equation}\label{eqn:jmmfcnl0later}
    F(\rho) =\int_{\Omega} \left(\frac{1}{2}(\rho-f)^2+\tau \rho \log \rho\right) \ud x,
\end{equation}
which has a unique minimizer. In the regime where the number of neurons approach infinity, the process of stochastic gradient descent is characterized by the following equation
\begin{equation}\label{eqn:jmmgfdeltalater}
    \partial_t \rho^\delta_t = \nabla \cdot (\rho_t^\delta \nabla \Psi) +\tau \Delta \rho_t^\delta, \, \textrm{ with }\Psi = -K^\delta*f + K^\delta*K^\delta*\rho_t^\delta,
\end{equation} which is the Wasserstein gradient flow of \eqref{eqn:jmmfunctionallater}. Heuristically, as $\delta \rightarrow 0$, the solution $\rho^\delta_t$ converges to the solution $\rho_t$ of the viscous porous-medium equation:
\begin{equation}\label{eq:pme}
    \partial_t \rho_t = -\nabla \cdot(\rho_t \nabla f) + \Delta \rho_t^2+  \tau \Delta \rho_t.
\end{equation}
Observe that equation \eqref{eq:pme} is the Wasserstein gradient flow of $F$.
\begin{theorem}\label{thm:jmmasyconv}
Let $\Omega\subset \R^d$ be a convex compact set with a $C^2$-boundary. Assume that $f\in C^\infty(\Omega;\R_+)$ and that $f$ is uniformly concave, i.e. there exists $\alpha>0$ such that $y^T D^2 f(x) y \leq -\alpha |y|^2$ for any $x\in \Omega$ and $y\in \R^d$.  Under the conditions of \cite{javanmard2020analysis}, let $\rho_t^\delta$ be the solution of \eqref{eqn:jmmgfdeltalater}. Then as $\delta\to 0$, the $\omega$-limit set of $\rho_t^\delta$ converges in Hausdorff metric with respect to $W_2$ to the unique minimizer of $F$ defined in \eqref{eqn:jmmfcnl0later}.
\end{theorem}
\begin{proof}
We need to verify all the conditions in Proposition \ref{prop:asyset}. The first condition (i) holds using Wasserstein displacement convexity of the energy functional \eqref{eqn:jmmfcnl0later}, condition (ii) can be proved easily using arguments from \cite[Theorem 4.1]{carrillo2019blob}, which is also done in \cite[Theorem 5.1]{craig2022blob}, and condition (iv) holds trivially on bounded domain. As for (iii), it is proven in \cite[Lemma E.2]{javanmard2020analysis} that $T_\delta^*=\infty$ for all $\delta$, and in \cite[Theorem 5.2]{javanmard2020analysis} that as $\delta \rightarrow 0$, $\rho^\delta(t) \xrightarrow{L^2} \rho(t)$ strongly for almost every $t$. The proof relies on showing the tightness of sequence $\{\rho_t^\delta\}_{t\in [0,T]}$ on the space $C([0,T]; \mathcal{P}(\Omega))$ as well as the uniqueness of the weak solution of \eqref{eq:pme}. Moreover, for the regularized energy \eqref{eqn:jmmfunctionallater} convergence as $\delta\to 0$ is proved in Lemma F.3 for the first term, and in the proof of Theorem F.8 for the entropy term. Therefore, we can appeal to Proposition \ref{prop:asyset} (ii) to prove that the asymptotic sets must also be consistent as $\delta \to 0$. 
\end{proof}

In our setting of kernelized birth-death dynamics, the assumptions (i), (ii), (iii) of Proposition \ref{prop:asyset}, with $d$ being the Wasserstein metric, are verified by Lemma \ref{lem:wellposeeps0} and Theorem \ref{thm:expconvbdkl}, Theorem \ref{thm:blobengyconv} and Theorem \ref{thm:bdgammaconv} (since $d_{SH}$ convergence implied convergence in Wasserstein metric) respectively, while assumption (iv) holds trivially on $\mathbb{T}^d$ due to compactness. Therefore we have the following theorem: 
\begin{theorem}\label{thm:bdasysetwkconv}
Under the assumptions of Theorem \ref{thm:bdgammaconv}, for any $\varepsilon$ and time sequence $(T_\varepsilon)_\varepsilon$, such that $\lim_{\varepsilon\to 0} T_\varepsilon = \infty$ and \eqref{eqn:gfkerbd} is well-posed up to time $T_\varepsilon$, we have \begin{equation*}
   \rhove_{T_\varepsilon} \xrightarrow{\varepsilon \to 0} \pi \text{ in }W_2, \text{ and } \lim_{\varepsilon\to 0}\fve(\rhove_{T_\varepsilon}) = 0.
\end{equation*}
\end{theorem}


\subsection{Particle based schemes} \label{sec:jump}

One possible idea for particle approximation is to consider particle solutions to 
\eqref{eqn:gfkerbd}, in analogy with the blob method for the Fokker-Planck equation \cite{carrillo2019blob}.
For discrete measure initial data $\sum_{i=1}^N m_i \delta_{x_i} $ the equation \eqref{eqn:gfkerbd} becomes an ODE system for the masses \eqref{eqn:bdkergfn}.
While formally this  provides a deterministic particle-based algorithm that converges to approximation of $\pi$, there are a number of challenges. Namely, the support of the measure does not change (i.e. particles do not move), 
and since masses of particles can become very uneven, this affects the quality of approximation. 

\smallskip

Instead of working to overcome these challenges (which remains an intriguing direction) we will consider a random, jump, particle process whose mean field limit is the equation \eqref{eqn:gfkerbd}. 
In the idealized birth-death dynamics with infinitely many particles \eqref{eqn:purebd} (with similar modifications to \eqref{eqn:gfchi2}), each particle has a jump rate \[\Lambda(x,\rho) = \frac{\delta \calF}{\delta \rho}-\int \frac{\delta \calF}{\delta \rho} \rho  = \log \left(\frac{\rho}{\pi} \right) - \int \log \left(\frac{\rho}{\pi} \right)  \rho.\] 
If $\Lambda>0$ then the particle jumps out of position $x$, and if $\Lambda<0$ then particle jumps into $x$, both with rate $|\Lambda|$. 
The issue with implementing such algorithm on the level of particle measures is that the pointwise density $\rho$ is not available. 

\smallskip

However if one considers the energy $\fve$ instead of $\mathcal F$, then the jump rates become 
\begin{align} \label{jumpe}
\begin{split}
\Lambdae(x,\rhove_t) = & \frac{\delta \fve}{\delta \rhove_t}-\int \frac{\delta \fve}{\delta \rhove_t} \rhove_t \\
=&  \log \left(\frac{\Kve*\rhove_t}{\pi}\right) + \Kve * \left( \frac{ \rhove_t}{\Kve* \rhove_t} \right) - \! \int \log \left(\frac{\Kve*\rhove_t}{\pi}\right) \rhove_t -1.
\end{split}
\end{align}
This  interpretation of \eqref{eqn:gfkerbd} allows us to construct a finite particle approximation of the dynamics \eqref{eqn:gfkerbd}, that is, if we let $\rhove_t = \frac{1}{N}\sum_{i=1}^N \delta_{x_t^i}$, then the particle at location $x_i$ is removed or added with rate \begin{align}\label{eqn:bdratekerkl}
\begin{split}
    \Lambda(x_t^i) = & \log\left( \frac{1}{N}\sum_{j=1}^N \Kve(x_t^i-x_t^j)\right) + \sum_{j=1}^N \frac{\Kve(x_t^i-x_t^j)}{\sum_{\ell=1}^N \Kve (x_t^j-x_t^{\ell})} - \log \pi(x_t^i) \\
    &- \frac{1}{N}\sum_{\ell=1}^N \log\left( \frac{1}{N}\sum_{j=1}^N \Kve(x_t^{\ell}-x_t^j)\right) -1 + \frac{1}{N}\sum_{\ell=1}^N \log \pi(x_t^\ell).
\end{split} 
\end{align}
This interpretation is where our sampling algorithm, as well as the one in \cite{lu2019accelerating}, are based on. To preserve the number of particles, at each birth-death step we uniformly add or remove another particle if the one in $x_i$ is remover or added, respectively.  

\smallskip

Before we discuss some properties and modifications to this jump dynamics, let us
 remark that there is an alternative way to create a jump process whose mean field limit, as $\varepsilon \to 0$ is expected to approach  the pure birth-death process, \eqref{eqn:purebd}, namely simply replacing $\rho$ by $\rho * K_\varepsilon$ in the rates for birth and death in \eqref{eqn:purebd}. This is the approach considered in \cite{lu2019accelerating}. The jump rates for such process are 
\begin{equation} \label{barjump}
    \overline{\Lambdae}(x,\rhove) = \log \left(\frac{\Kve*\rhove_t}{\pi} \right) -\int_{\R^d}\log \frac{\Kve*\rhove_t}{\pi} \rhove_t 
\end{equation} 
The expected mean field limit for fixed $\varepsilon>0$ would be
 \begin{equation} \label{eqn:kerbdnongf}
 \partial_t \rhove_t = -\rho_t\left(\log(\Kve*\rhove_t)-\log \pi-\int_{\R^d} \log \frac{\Kve*\rhove_t}{\pi} \rhove_t \right).
 \end{equation} 
A downside of the dynamics \eqref{eqn:kerbdnongf} is that it is unclear if it possesses a gradient flow structure or  a Lyapunov functional that approximates KL divergence, which is why we are modifying the jump rates to be \eqref{eqn:bdratekerkl}. 

\smallskip

An alternative ensemble based sampling, where the birth-death process was achieved via jumps, has recently been introduced and studied in \cite{lindsey2021ensemble}, where the 
jump rate was 
\begin{equation}\label{eqn:lwzratechi2}
\Theta_\varepsilon = \frac{K_\varepsilon * \rhove}{\pi} - \int \frac{K_\varepsilon * \rhove}{\pi} \rhove.
\end{equation}
The limit of the mean field dynamics as $\varepsilon \to 0$ is the spherical Hellinger gradient flow of the $\chi^2$ divergence, \eqref{eqn:gfchi2}. In particular the birth-death part of their dynamics relates to \eqref{eqn:gfchi2} in the same way as the dynamics of \eqref{eqn:kerbdnongf} relates to Hellinger gradient flow of KL divergence, \eqref{eqn:purebd}. Unlike our choice of \eqref{eqn:bdratekerkl}, the rate in \eqref{eqn:lwzratechi2} does depend on the normalization constant of $\pi$, which \cite{lindsey2021ensemble} can avoid by a rescaling of time.

\smallskip

We note that a serious issue with jump processes discussed above is that the support of the measure is not expanding. The jumps only lead to new particles at the occupied locations. 
In \cite{lindsey2021ensemble} this issue is dealt with as follows: each particle created at some $x$  is moved to proposal obtained by sampling $K_\varepsilon( \, \cdot\, - x)$. This proposal is accepted according to the standard Metropolis procedure, thus ensuring that one is sampling the probability measure $\pi$. 
In our numerical experiments in Section \ref{sec:numerics} we combine the jump process of $\Lambdae$ with unadjusted Langevin algorithm (ULA). The latter sampler is responsible for exploring new territory, especially high density regions in the state space. In effect we move each particles by a gradient descent plus sampling the Gaussian centered at the particle. We did not add a Metropolis step in our experiments.

\section{Numerical Examples} \label{sec:numerics}

\begin{example}[A toy example] \label{exp:toyGMM}
This example is a modification of \cite{lu2019accelerating}*{Example 2}. We consider a two-dimensional Gaussian mixture model with four components, i.e. $\pi(x,y) = \sum_{i=1}^4 w_i \times  \mathcal{N}(m_i,\Sigma_i)$ and initial particles sampled from Gaussian $\mathcal{N}(m_0,\Sigma_0)$ where the parameters are given by  \begin{align*}
    [w_1, w_2 , w_3 , w_4] = [0.5 , 0.1 , 0.1 , 0.3], \, m_1 =  [0,2], \, m_2=[-3,5], \, m_3=[0,8], \, m_4=[3,5]. \\  \Sigma_1 = \Sigma_3 = \begin{pmatrix} 0.8 & 0 \\ 0 & 0.01 \end{pmatrix}, \, \Sigma_2=\Sigma_4 = \begin{pmatrix}0.01 & 0 \\ 0 & 1 \end{pmatrix}, \, m_0 = [0,8], \, \Sigma_0 = \begin{pmatrix} 0.3 & 0 \\ 0 & 0.3 \end{pmatrix}.
\end{align*}
Morally speaking, each of the four Gaussian components of $\pi$ are essentially supported on a very narrow domain with little intersection between  each other. At the beginning, all particles are concentrated near the top Gaussian centered at $m_0=m_3$, which is a metastable region, so if the particles follow the overdamped Langevin dynamics, it will take an extremely long time for each particle to escape any certain Gaussian, and with many particles, it is numerically intractable to observe a significant amount of particles to be present in all metastable regions.

\smallskip

Our algorithm BDLS-KL is an implementation on a modification of \eqref{eqn:gfkerbd}, that is, \begin{equation}\label{eqn:gfkerbdfk}
    \partial_t \rhove_t = \nabla \cdot (\rhove_t \nabla \log \frac{\rhove_t}{\pi})-\rhove_t \left(\log(\frac{\Kve*\rhove_t}{\pi}) + \Kve * \left( \frac{ \rhove_t}{\Kve* \rhove_t} \right) - \int \rhove_t\log(\frac{\Kve*\rhove_t}{\pi})-1 \right).
\end{equation} We simulate \eqref{eqn:gfkerbdfk} using a ``splitting scheme'', that is alternating between an unadjusted Langevin step and a birth-death step. More precisely, we use approximate the density $\rho_t^{(\varepsilon)}$ with a finite sum of Diracs with equal weights, i.e. $\rhove_t \approx \frac{1}{N}\sum_{i=1}^N \delta_{x_t^{(i)}}$, and at each time step we first perform a Langevin move for all particles and then a birth-death move with jump rates given by \eqref{eqn:bdratekerkl}. The Fokker-Planck term is necessary in our algorithm due to the fact that the pure birth-death dynamics does not find new locations. For the algorithm BDLS-chi2, we replace the energy functional by regularized $\chi^2$-divergence, that is $\frac{1}{2}\int \rho\frac{\Kve*\rho}{\pi}$, and everything else is identical to BDLS-KL.

\smallskip
We compare these two birth-death based sampling methods with the unadjusted overdamped Langevin dynamics (ULA) as well as SVGD \cite{liu2016stein}. We choose $N=800$ particles, $\Delta t =  10^{-3}$ and kernel bandwidth $\varepsilon=0.2$ for all algorithms and compare their error of estimating $\mathbb{E}_\pi f$ with $f(x,y) = x^2/3+y^2/5$, as well as their Maximum Mean Discrepancy (MMD) \cite{arbel2019maximum}, which can be computed explicitly for Gaussian mixtures.

\smallskip

Figure \ref{fig:gmmquantities} shows that the algorithms based on birth-death dynamics converge to equilibrium much faster than Langevin dynamics or SVGD in terms of both observable error and MMD distance. More precisely, algorithms based on birth-death dynamics reach equilibrium at $T\approx 10$, while for the other two algorithms, although they also eventually converge to equilibrium, it is not achieved even at $T\approx 100$. Figure \ref{fig:gmmparticles} provides a more intuitive explanation on the fact that birth-death based algorithms are significantly better at penetrating energy barriers and overcoming metastability. 
\begin{figure}
    \begin{minipage}{2in}
    \includegraphics[width=1.85in]{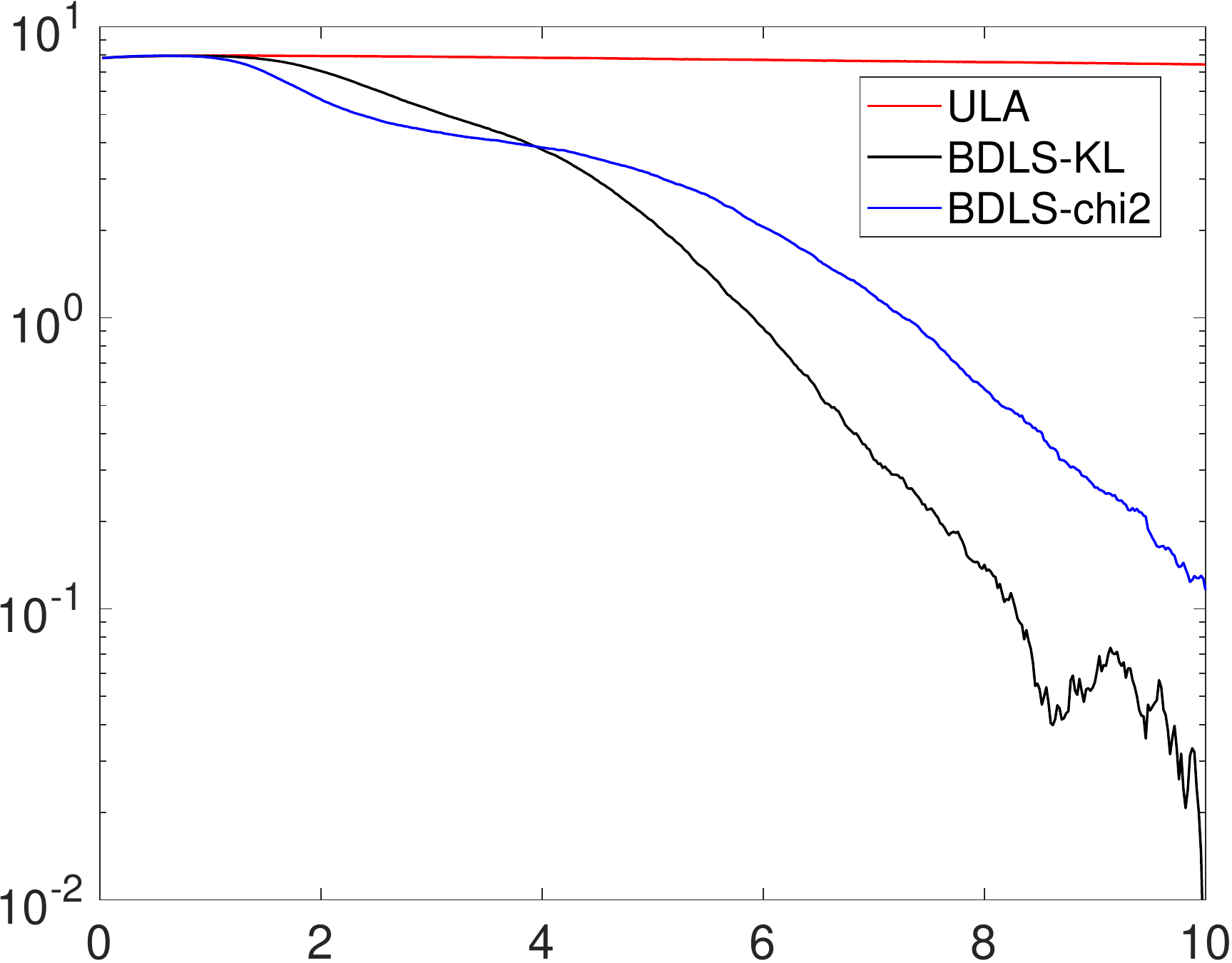}
    \end{minipage}
    \begin{minipage}{2in}
    \includegraphics[width=1.85in]{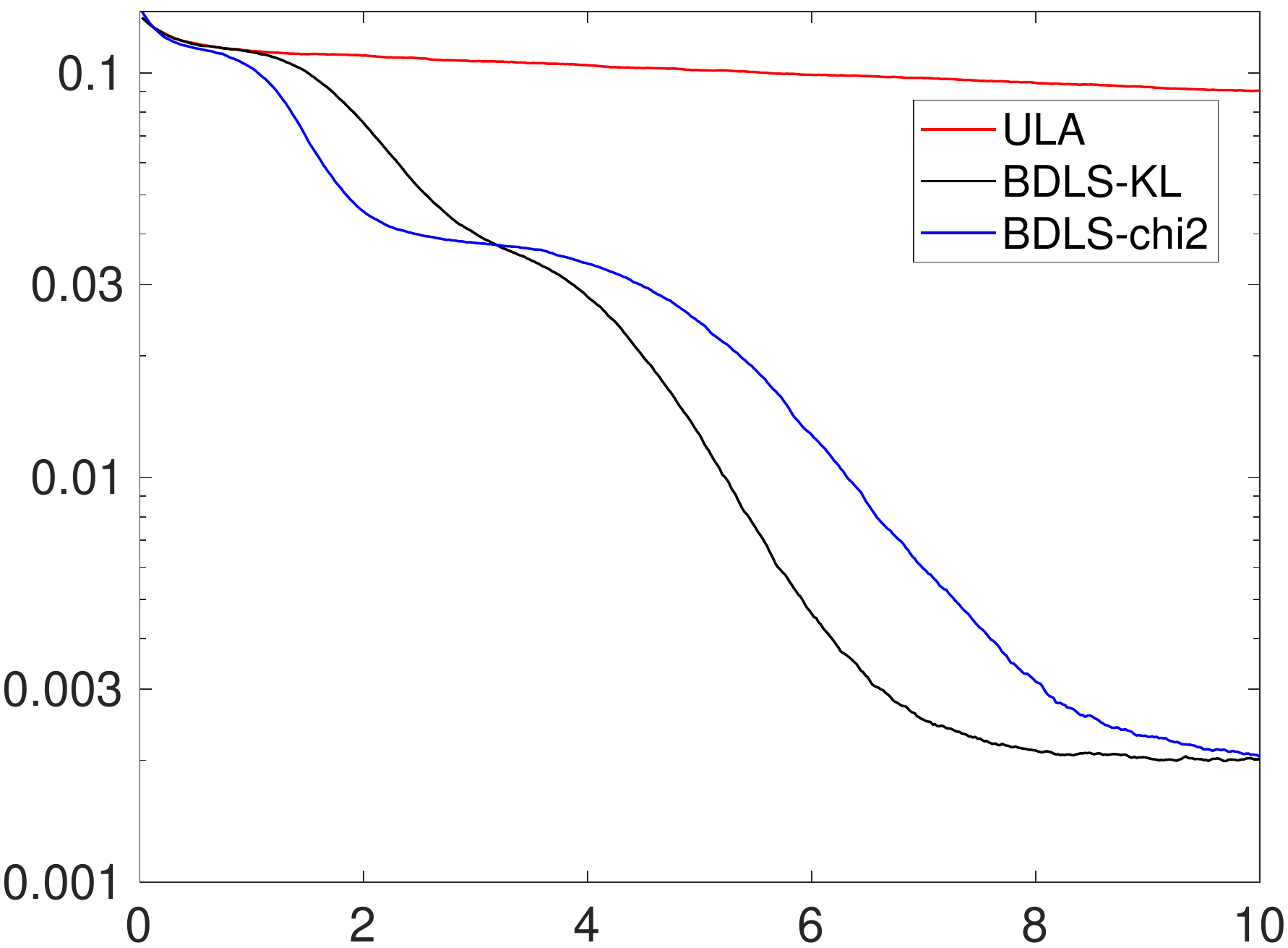}
    \end{minipage}
    \begin{minipage}{2in}
    \includegraphics[width=1.85in]{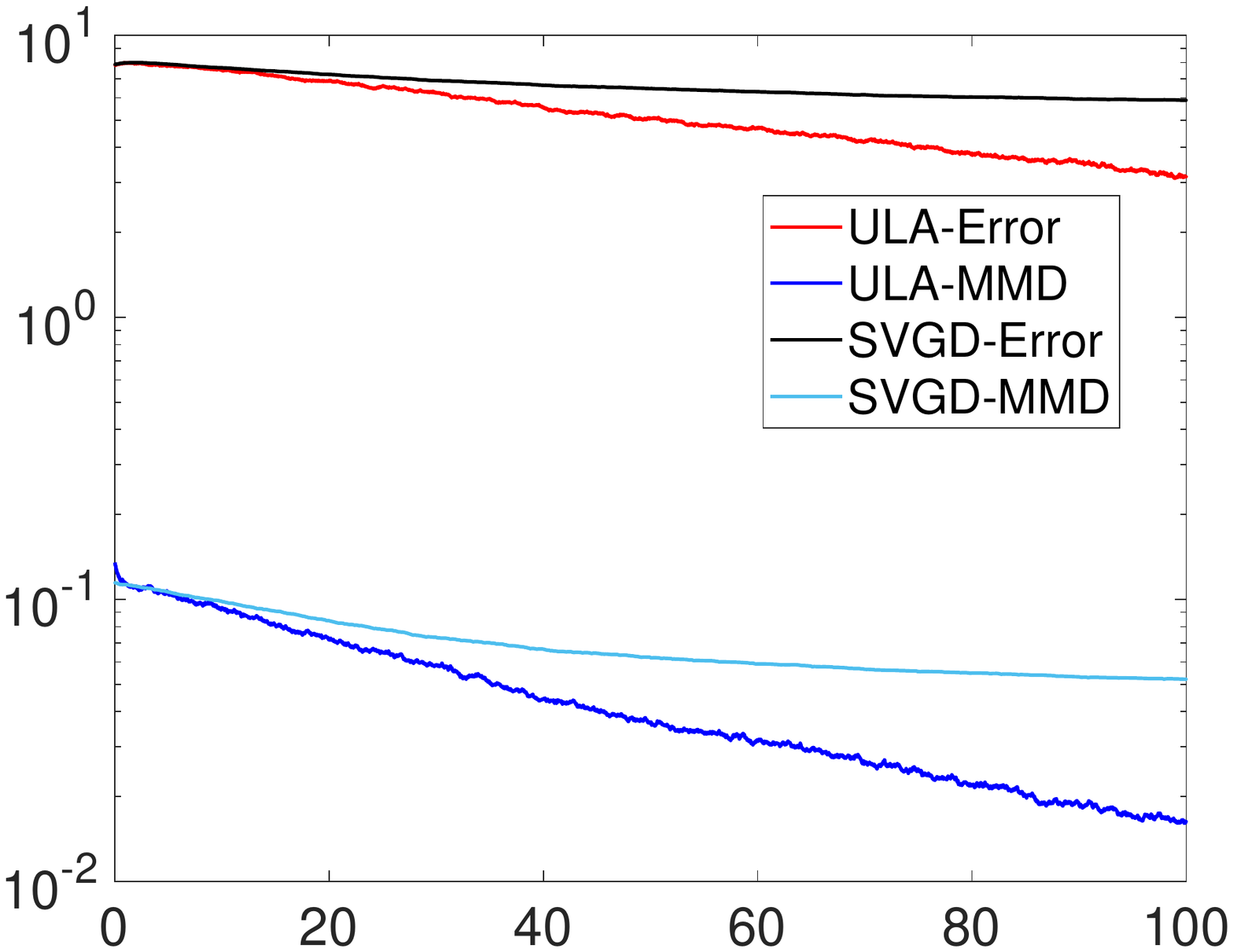}
    \end{minipage}
    \caption{Gaussian mixture example. Left: error of observable $f(x,y) = x^2/3+y^2/5$; center: MMD with kernel $K(x,y) = (2\pi)^{-\frac{d}{2}}e^{-\frac{|x-y|^2}{2}}$; right: observable error and MMD for Langevin dynamics (ULA) and SVGD up to $T=100$. Both left and center plots are averaged over 30 experiments. Both birth-death algorithms based on KL and $\chi^2$ converge much faster to equilibrium as $t$ gets larger.}
    \label{fig:gmmquantities}
\end{figure}
\begin{figure}
    \begin{minipage}{5in}
    \centering
    \includegraphics[width=4.9in]{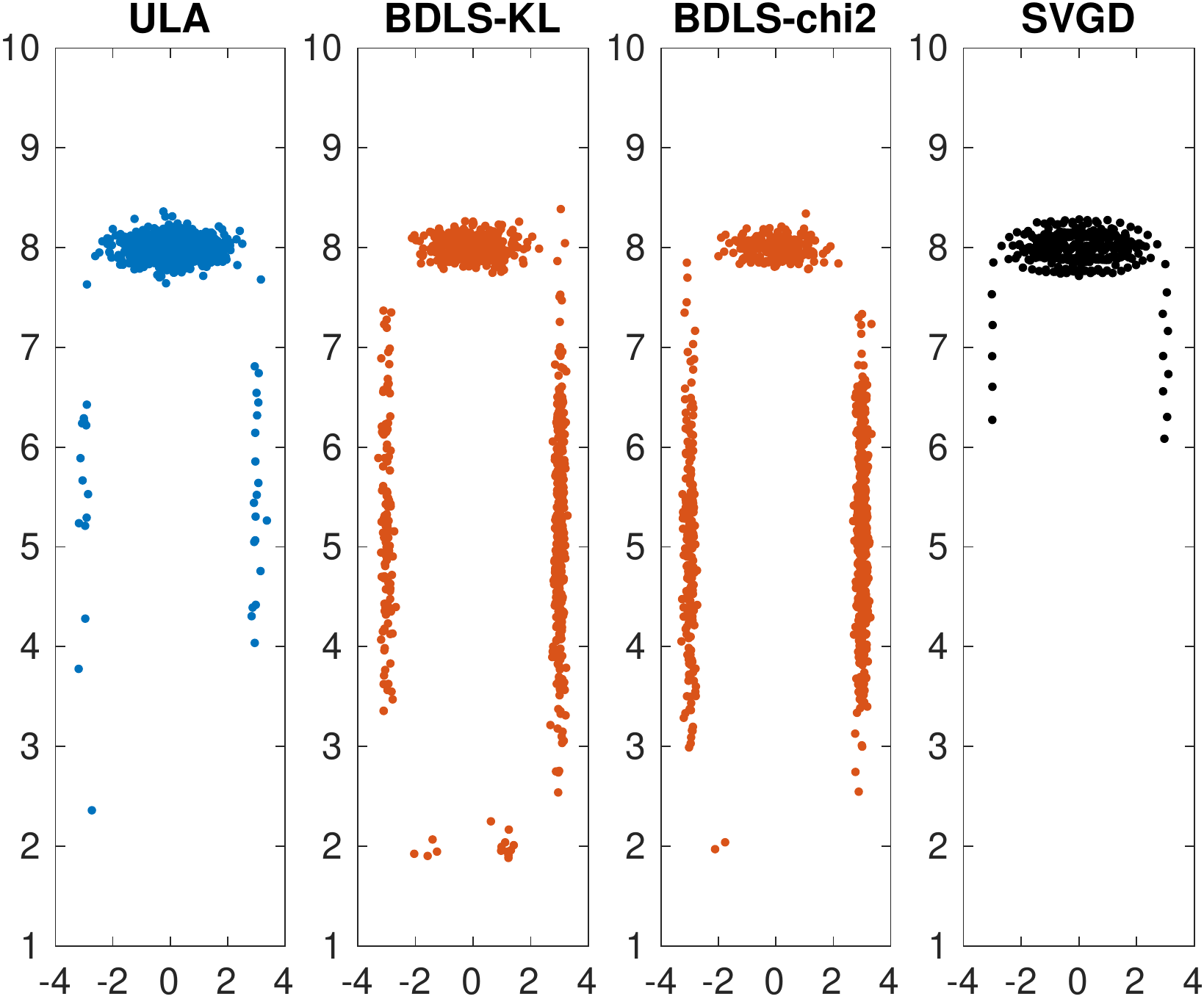}
    \end{minipage}
    \begin{minipage}{5in}
    \centering
    \includegraphics[width=4.9in]{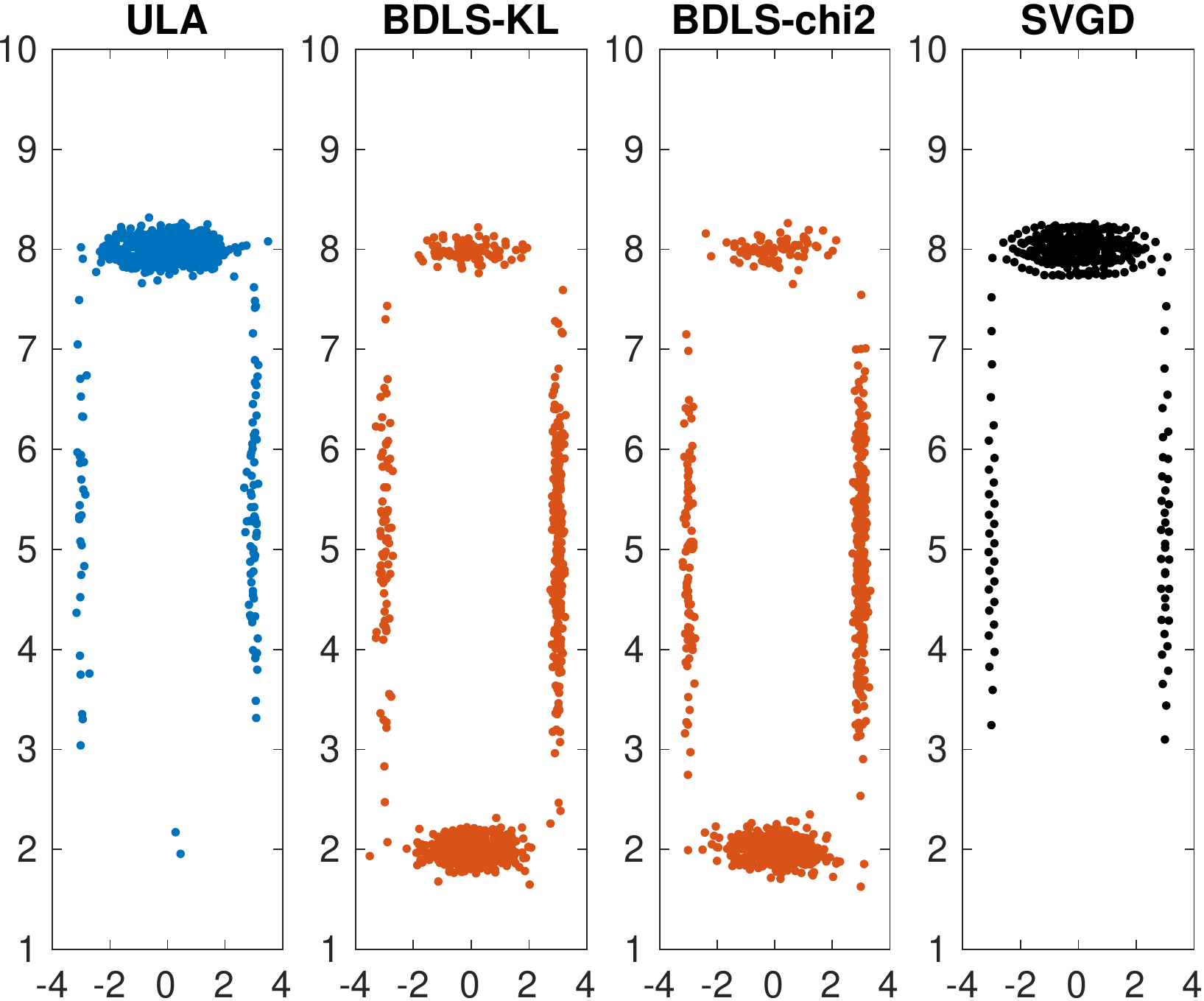}
    \end{minipage}
    \caption{Gaussian mixture example. Top: position of particles at $T=3$; bottom: position of particles at $T=10$. Algorithms based on birth-death are better at attracting particles into under-explored regions.}
    \label{fig:gmmparticles}
\end{figure}

\end{example}

\begin{example}[Real-world data] \label{exp:BayesClass}
We also tested the birth-death method on the Bayesian logistic regression for binary classification
using the same setting as \cite{gershman2012nonparametric, liu2016stein, korba2021kernel}, which assigns the hidden regression weights $w$ with a precision parameter $\alpha\in \R_+$, and that we impose Gaussian prior $p(w|\alpha) = \mathcal{N} (w, \alpha^{-1}\mathrm{Id})$ on $w$ and $p(\alpha) = \textrm{Exp}(0.01)$. The observables $y\in \{-1,1\}$ are generated by $\mathbb{P}(y=1|x,w) = (1+\exp(-w^\top x))^{-1}$. The inference is applied on posterior $p([w, \log \alpha]|[x,y])$. We compare the performance of our algorithm with SVGD in terms of accuracy and log-likelihood. We would like to comment here that the kernel bandwidth of SVGD is chosen using the median trick, while for birth-death, since the bandwidth is proportional to the bias, we choose the bandwidth to be $0.1$, $0.5$ or $1$, whichever provides the best performance. The results are shown in Table \ref{table:BDLSvsSVGD}, showing that when the dimension is not too large and running time is relatively long, both birth-death sampler (with Langevin steps) and SVGD perform similarly well. 

\smallskip

We also compare the behavior of both criteria, accuracy and log-likelihood, as time evolves between $t\in[0,5]$. The results are shown in Figure \ref{fig:benchmarktime}. One can observe that birth-death Langevin sampler reaches the desired accuracy at an extremely short time $t\approx 0.3$, which SVGD cannot achieve before $t=5$. This indicates that one can run BDLS for a much shorter time to reach equilibrium, significantly alleviating the computational cost issue of running a system of many particles. This is the spirit of our Theorem \ref{thm:expconvbdkl}.

\smallskip

We would like to comment that for the dataset ``splice'' where $d=61$, the performance of birth-death sampler is significantly worse, even with $N=2000$, which is not entirely surprising due to the limitations of kernel density estimation. It is an interesting open question to design a sampling algorithm which can inherit the fast convergence properties of spherical Hellinger gradient flows and be robust in high dimension settings.
\begin{table}
\begin{tabular}{c|c| c|c|c|c}
     \hline Dataset & Dimension &  \makecell{BDLS \\ accuracy} & \makecell{BDLS \\ log-likelihood} & \makecell{SVGD \\ accuracy} & \makecell{SVGD \\ log-likelihood}  \\ \hline
     Banana & 3 & 0.583 & -0.690 & 0.585 & -0.686 \\ Breast\_cancer & 10 & 0.714 & -0.604 & 0.714 & -0.586 \\ Diabetis & 9 &  0.763 & -0.527 & 0.753 & -0.529 \\ Flare\_solar & 10 & 0.683 & -0.578 & 0.685 & -0.600 \\ German & 21 & 0.687 & -0.598 & 0.680 & -0.597 \\ Heart & 14 & 0.840 & -0.376 & 0.850 & -0.379 \\ Image & 19 & 0.817 & -0.433 &  0.815 & -0.434 \\ Ringnorm & 21 & 0.760 & -0.521 & 0.760 & -0.501 \\ Thyroid & 6 & 0.933 & -0.250 & 0.933 & -0.294 \\  Titanic & 4 & 0.780 & -0.586 & 0.785 & -0.566 \\ Twonorm & 21 & 0.973 & -0.069 & 0.975 & -0.112 \\ Waveform & 22 & 0.773  & -0.466 & 0.773 & -0.469
\end{tabular}
\caption{Bayesian logistic regression for binary classification. For both algorithms, $N=500$, time stepsize $\Delta t=10^{-3}$, final time $T=15$.}
\label{table:BDLSvsSVGD}
\end{table}

\begin{figure}
    \begin{minipage}{3in}
    \includegraphics[width=2.9in]{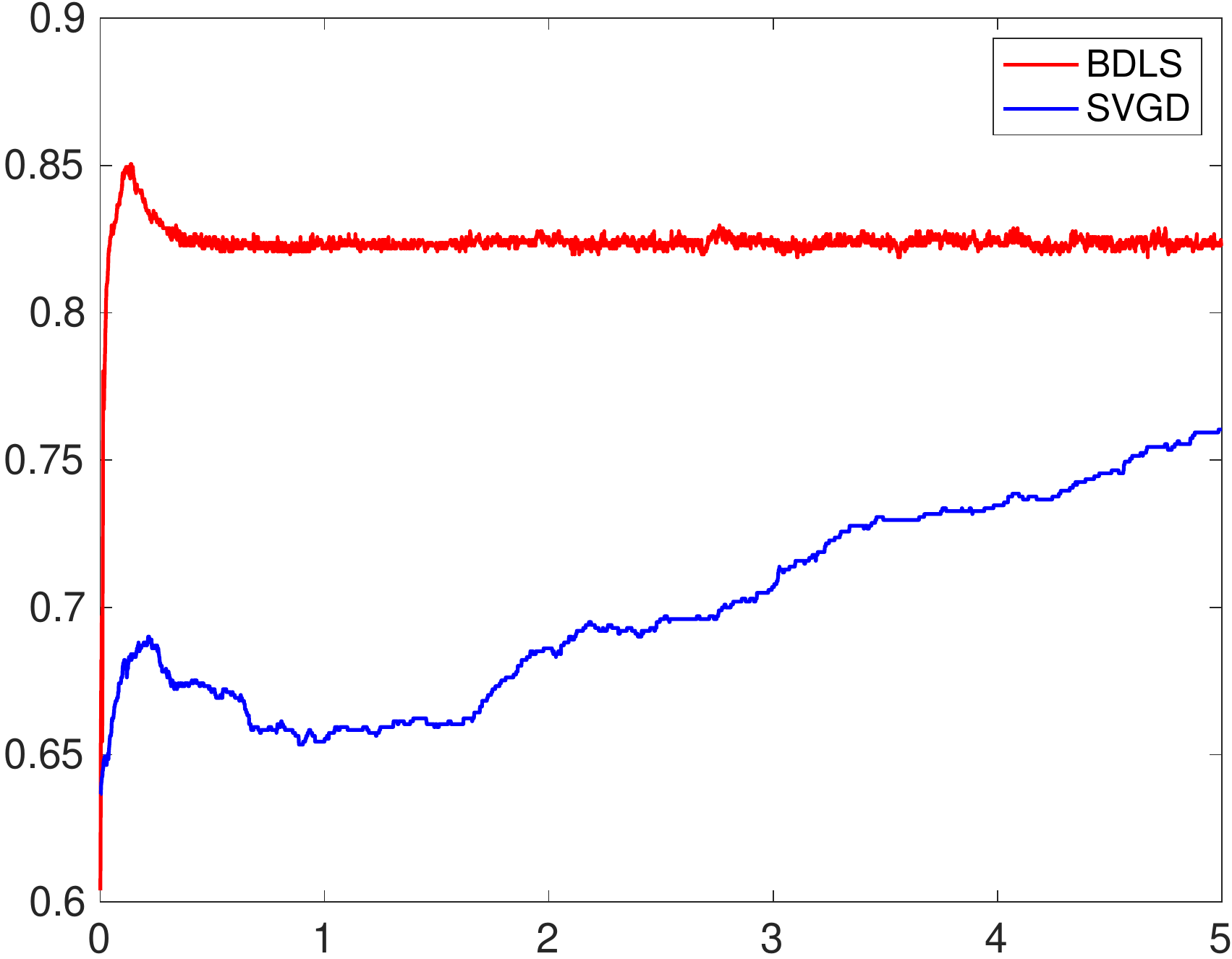}
    \end{minipage}
    \begin{minipage}{3in}
    \includegraphics[width=2.9in]{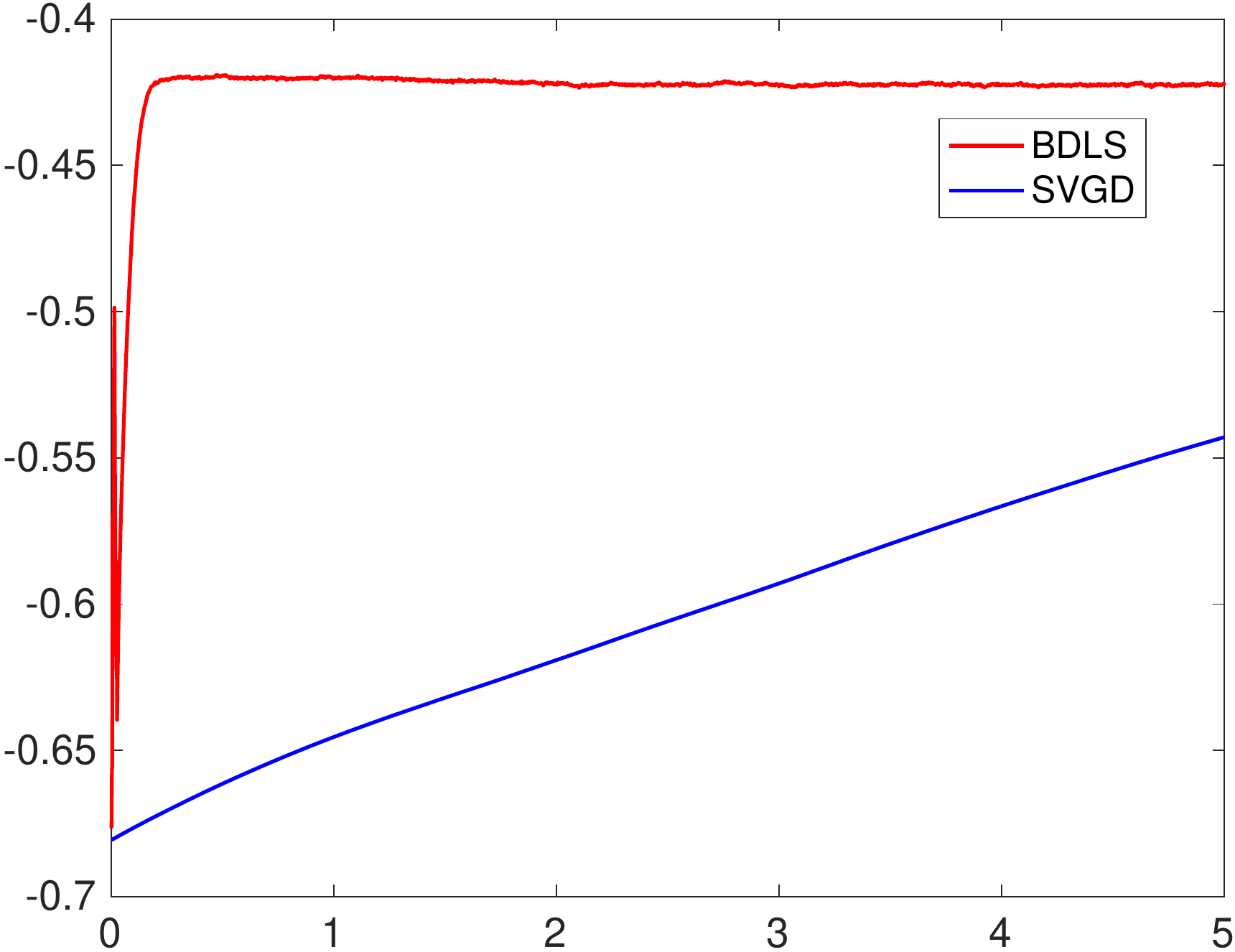}
    \end{minipage}
    \caption{Bayesian classification problem with dataset ``Image''. The birth-death Langevin algorithm reaches the desired accuracy and log-likelihood much faster than SVGD.}
    \label{fig:benchmarktime}
\end{figure}
\end{example}

\section*{Acknowledgement}
The authors would like to thank Katy Craig and L\'eonard Monsaingeon for helpful discussions. LW would like to thank Rishabh Gvalani for discussions during revision. YL thanks NSF for the support via the award DMS-2107934. DS is grateful to NSF for support via grant DMS 2206069.  Part of this work was done while the authors were visiting the Simons Institute for the Theory of Computing. The authors are thank the institute for hospitality. DS and LW would also like to thank CNA at CMU for support. 

\section*{Data availability statement}
The data that support the findings of this study are available upon request from the authors.

\bibliographystyle{amsplain}
\bibliography{bd}

\end{document}